\scrollmode

\documentclass{amsart}
\usepackage{amssymb}
\usepackage[all]{xy}

\theoremstyle{plain}
\newtheorem{prop}{Proposition}
\newtheorem{thm}[prop]{Theorem}
\newtheorem{cor}[prop]{Corollary}
\newtheorem{lem}[prop]{Lemma}
\newtheorem{fact}[prop]{Fact}

\newtheorem*{thmA}{Theorem A}
\newtheorem*{corB}{Corollary B}
\newtheorem*{thmC}{Theorem C}
\newtheorem*{thmD}{Theorem D}
\newtheorem*{thmE}{Theorem E}

\theoremstyle{definition}

\theoremstyle{remark}

\newtheorem{rem}[prop]{Remark}

\numberwithin{prop}{section} 
\numberwithin{equation}{section}
\numberwithin{step}{prop}

\DeclareMathOperator{\op}{op}
\DeclareMathOperator{\ob}{ob}
\DeclareMathOperator{\mor}{mor}
\DeclareMathOperator{\iid}{id}
\DeclareMathOperator{\nat}{nat}
\DeclareMathOperator{\ext}{ext}
\DeclareMathOperator{\Hom}{Hom}
\DeclareMathOperator{\Ext}{Ext}

\DeclareMathOperator{\quot}{quot}
\DeclareMathOperator{\rk}{rk}

\DeclareMathOperator{\gr}{gr}

\DeclareMathOperator{\sat}{sat}
\DeclareMathOperator{\tor}{tor}
\DeclareMathOperator{\hd}{hd}
\DeclareMathOperator{\soc}{soc}

\DeclareMathOperator{\burn}{B}

\DeclareMathOperator{\Tot}{Tot}
\DeclareMathOperator{\fg}{f.g.}
\DeclareMathOperator{\rad}{rad}

\DeclareMathOperator{\glatdim}{Ldim}
\DeclareMathOperator{\gldim}{gldim}
\DeclareMathOperator{\prdim}{proj.dim}

\DeclareMathOperator{\kernel}{ker}
\DeclareMathOperator{\image}{im}
\DeclareMathOperator{\coker}{coker}

\DeclareMathOperator{\rst}{res}
\DeclareMathOperator{\ifl}{inf}
\DeclareMathOperator{\dfl}{def}

\DeclareMathOperator{\Midn}{\mathbf{ind}}

\newcommand{\hH}{\hat{H}}

\newcommand{\ttau}{\tilde{\tau}}
\newcommand{\tttau}{\hat{\tau}}
\newcommand{\ca}[1]{\mathcal{#1}}

\newcommand{\F}{\mathbb{F}}
\newcommand{\N}{\mathbb{N}}
\newcommand{\Z}{\mathbb{Z}}

\newcommand{\RR}{\mathbb{R}}
\newcommand{\II}{\mathbb{I}}

\newcommand{\RG}{R[G]}

\newcommand{\Mod}{\mathbf{mod}}
\newcommand{\lat}{\mathbf{lat}}

\newcommand{\caC}{\ca{C}}
\newcommand{\caD}{\ca{D}}
\newcommand{\RMod}{{}_R\Mod}
\newcommand{\FMod}{{}_{\F}\Mod}
\newcommand{\Rlat}{{}_R\lat}
\newcommand{\euF}{\mathfrak{F}}

\newcommand{\MC}{\mathfrak{cMF}}
\newcommand{\Gent}{\ca{G}}
\newcommand{\mf}{\ca{M}}

\newcommand{\eut}{\mathfrak{t}}
\newcommand{\eui}{\mathfrak{i}}

\newcommand{\boF}{\mathbf{F}}
\newcommand{\boG}{\mathbf{G}}
\newcommand{\boP}{\mathbf{P}}
\newcommand{\boJ}{\mathbf{J}}
\newcommand{\boX}{\mathbf{X}}
\newcommand{\tboX}{\tilde{\boX}}
\newcommand{\boY}{\mathbf{Y}}
\newcommand{\boZ}{\mathbf{Z}}
\newcommand{\boS}{\mathbf{S}}
\newcommand{\boR}{\mathbf{R}}
\newcommand{\buT}{\boldsymbol{\Upsilon}}
\newcommand{\boT}{\mathbf{T}}
\newcommand{\boM}{\mathbf{M}}
\newcommand{\bSigma}{\boldsymbol{\Sigma}}
\newcommand{\boB}{\mathbf{B}}
\newcommand{\boC}{\mathbf{C}}
\newcommand{\boK}{\mathbf{K}}
\newcommand{\boQ}{\mathbf{Q}}
\newcommand{\boI}{\mathbf{I}}
\newcommand{\boh}{\mathbf{h}}
\newcommand{\caR}{\ca{R}}

\newcommand{\caL}{\ca{L}}

\newcommand{\boc}{\mathbf{c}}
\newcommand{\bok}{\mathbf{k}}

\newcommand{\tchi}{\tilde{\chi}}
\newcommand{\tbeta}{\tilde{\beta}}

\newcommand{\miniblacksquare}{\scriptscriptstyle{\blacksquare}}
\newcommand{\eps}{\varepsilon}
\newcommand{\der}{\partial}
\newcommand{\argu}{\underline{\phantom{x}}}

\begin{document}
\title[Lattices and cohomological Mackey functors]{Lattices and cohomological Mackey functors for finite cyclic p-groups}
\author{B. Torrecillas
and Th. Weigel}

\date{\today}

\address{B. Torrecillas\\
Departamento de Algebra y An\'alisis Matem\'atico\\
Universidad de Almer\'ia\\
04071 Almer\'ia, Spain}
\email{ btorreci@ual.es}

\address{Th. Weigel\\
Universit\`a di Milano-Bicocca\\
U5-3067, Via R.Cozzi, 53\\
20125 Milano, Italy}
\email{thomas.weigel@unimib.it}

\thanks{The first author is supported by the grants 
MTM2011-27090 from the ``Ministerio de Ciencia
e Innovacion''  and P07-FQM-0312 from the ``Junta de Andaluca'' (FEDER).}

\begin{abstract}
For a finite cyclic $p$-group $G$ and a discrete valuation domain
$R$ of characteristic $0$ with maximal ideal $pR$
the $\RG$-permutation modules are characterized
in terms
of the vanishing of first degree cohomology on all subgroups
(cf. Thm.~A). As a consequence any $\RG$-lattice can be presented
by $\RG$-permutation modules (cf. Thm.~C).
The proof of these results is based on 
a detailed analysis of the category of cohomological
$G$-Mackey functors with values in the category of $R$-modules.
It is shown that this category has global dimension $3$
(cf. Thm.~E). A crucial step in the proof of Theorem E is the fact
that a gentle $R$-order category (with parameter $p$)
has global dimension less or equal to $2$
(cf. Thm.~D).
\end{abstract}


\maketitle


\section{Introduction}
\label{s:intro}
For a Dedekind domain $R$ and a finite group $G$ one calls
a finitely generated left $R[G]$-module $M$ an {\it $R[G]$-lattice}, if
$M$ - considered as an $R$-module - is projective.
In this paper we focus on the study of $R[G]$-lattice, where $R$
is a discrete valuation domain of characteristic $0$ with maximal ideal $pR$ for some
prime number $p$, and $G$ is a finite cyclic $p$-group.
The study of such lattices has a long history and was motivated by a 
promissing result of
F.-E.~Diederichsen (cf. \cite[Thm.~34:31]{CR:met1}, \cite{died:ham}) 
who showed that for the finite cyclic group of order $p$
there are precisely three directly indecomposable such lattices 
up to isomorphism: the trivial $R[G]$-lattice $R$,
the free $R[G]$-lattice $R[G]$, and the augmentation ideal
$\omega_{\RG}=\kernel(R[G]\to R)$.
A similar finiteness result holds for cyclic groups of order $p^2$
(cf. \cite{hr:rep1}).
However, for cyclic $p$-groups of order larger than $p^2$
there will be infinitely many isomorphism types of such lattices;
even worse, in general this classification problem is ``wild''
(cf. \cite{diet:rep1}, \cite{diet:rep2}, \cite{gud:wild}).
If the $R[G]$-lattice $M$ is isomorphic to $R[\Omega]$
for some finite left $G$-set $\Omega$, $M$ 
will be called an {\it $R[G]$-permutation lattice}.
The main purpose of this paper is to establish the following characterization
of $R[G]$-permutation lattices for finite cyclic $p$-groups
(cf. Cor.~\ref{cor:hil90lat1}, Prop.~\ref{prop:elequi}).

\begin{thmA}
Let $R$ be a discrete valuation domain of characteristic $0$ 
with maximal ideal $pR$ for some prime number $p$, 
let $G$ be a finite cyclic $p$-group, and let $M$ be an $R[G]$-lattice.
Then the following are equivalent.
\begin{itemize}
\item[(i)] $M$ is an $R[G]$-permutation lattice,
\item[(ii)] $H^1(U,\rst^G_U(M))=0$ for all subgroups $U$ of $G$,
\item[(iii)] $M_U$ is $R$-torsion free for all subgroups $U$ of $G$,
\end{itemize}
where $M_U=M/\omega_{R[U]}M$ denotes the $U$-coinvariants of $M$.
\end{thmA}

By a result of I.~Reiner (cf.  \cite[Thm.~34.31]{CR:met1}, \cite{rei:intcyc}), one knows that there are
$\Z[C_p]$-lattices satisfying (ii),
where $C_p$ is the cyclic group of order $p$,
which are not $\Z[C_p]$-permutation lattices.
Hence the conclusion of Theorem A does not hold for the ring $R=\Z$.

Theorem A has a number of interesting consequences which we would like to
explain in more detail. 
For a finite $p$-group $G$ it is in general quite difficult to decide whether
a given $\RG$-lattice $M$ is indeed an 
$\RG$-permutation lattice.
A sufficient criterion to the just mentioned problem
was given by A.~Weiss in \cite{weiss:rig} for an arbitrary
finite $p$-group $G$ and the ring of $p$-adic integers $R=\Z_p$.
He showed that if for a normal subgroup $N$ of $G$ 
the $\Z_p[G/N]$-module $M^N$ of $N$-invariants is a
$\Z_p[G/N]$-permutation module,
and $\rst^G_N(M)$ is a free
$\Z_p[N]$-module, then $M$ is a $\Z_p[G]$-permutation
module (cf. \cite[Chap.~8, Thm.~2.6]{karp:ind}).
Theorem A extends A.~Weiss' result
for cyclic $p$-groups in the following way (cf. Prop.~\ref{prop:exweiss}).

\begin{corB}
Let $R$ be a discrete valuation domain of characteristic $0$ 
with maximal ideal $pR$ for some prime number $p$, 
let $G$ be a finite cyclic $p$-group, and let $N$ be a normal subgroup
of $G$. Suppose that the $R[G]$-lattice $M$ is satisfying the following
two hypothesis.
\begin{itemize}
\item[(i)] $\rst^G_N(M)$ is an $R[N]$-permutation module, and
\item[(ii)] $M^N$ is an $R[G/N]$-permutation module.
\end{itemize} 
Then $M$ is an $\RG$-permutation module.
\end{corB}

Although it seems impossible to describe
all isomorphism types of directly indecomposable $R[G]$-lattices,
where $R$ is a discrete valuation domain of characteristic $0$ 
with maximal ideal $pR$ and
$G$ is a finite cyclic $p$-group,
one can (re)present such lattices in a very natural way (cf. Thm.~\ref{thm:preslat}).

\begin{thmC}
Let $R$ be a discrete valuation domain of characteristic $0$ 
with maximal ideal $pR$ for some prime number $p$, 
let $G$ be a finite cyclic $p$-group, and let $M$ be an $R[G]$-lattice.
Then there exist finite $G$-sets $\Omega_0$ and $\Omega_1$,
and a short exact sequence
\begin{equation}
\label{eq:preslat}
\xymatrix{
0\ar[r]& R[\Omega_1]\ar[r]&R[\Omega_0]\ar[r]&M\ar[r]&0
}
\end{equation}
of $R[G]$-lattices.
\end{thmC}

The proof of Theorem A and Theorem C is based on the theory of
{\it cohomological Mackey functors} for a finite group $G$.
Mackey functors were first introduced by
A.W.M.~Dress in \cite{dress:mac}. Cohomological Mackey functors
satisfy an additional identity (cf. \cite{pw:user}).
The category of 
cohomological $G$-Mackey functors $\MC_G(\RMod)$ with values in the category
of $R$-modules coincides with the category
of contravariant functors of an $R^\circledast$-order category $\mf_R(G)$ (cf. \S\ref{ss:maccat}). 
In case that $G$ is a cyclic $p$-group or order $p^n$, one has a 
{\it unitary projection functor} (cf. \S\ref{ss:funRcat})
\begin{equation}
\label{eq:unipro}
\pi\colon\MC_R(G)\longrightarrow\Gent_R(n,p)
\end{equation}
which can be used to analyze the category $\MC_G(\RMod)$.
Here $\Gent_R(n,p)$ denotes the {\it gentle $R$-order category}
supported on $n+1$ vertices and parameter $p$ (cf. \S\ref{ss:gentordcat})
which can be seen
as an $R$-order version of the gentle algebra $\Gent_\F(n)$ defined over a field $\F$.
The gentle algebra has been
subject to intensive investigations (cf. \cite{gere:gent}), e.g.,
it is well known that $\Gent_\F(n)$ is $1$-Gorenstein 
(resp. $0$-Gorenstein for $n=0$ or $n=1$), but for $n\geq 1$
it is not of finite global dimension.
Hence the following property  
of the gentle $R$-order category is somehow surprising (cf. Thm.~\ref{thm:gldimgent}).

\begin{thmD}
Let $p$ be a prime number, and let
$R$ be a principal ideal domain of characteristic $0$
such that $p.1\in R$ is a prime element. Then
$\gldim(\Gent_R(n,p))= 2$ for $n\geq 2$,
and $\gldim(\Gent_R(n,p))= 1$ for $n=1$ or $0$.
\end{thmD}

In \S\ref{s:secmac} we study {\it section cohomology groups}
which can be associated to any cohomological Mackey functor and any normal
section of a finite group. This allows us to introduce the notion of 
cohomological Mackey functors with the
{\it Hilbert$^{90}$-property}  (cf. \S\ref{ss:h0}).
Theorem~A and Theorem C are a direct consequence
of a more general result which states that 
for a discrete valuation domain $R$ of characteristic~$0$
and maximal ideal $pR$
every cohomological $G$-Mackey functor
with values in the category of $R$-lattices and with the Hilbert$^{90}$ property is projective
(cf. Thm.~\ref{thm:defcat}). 
The proof of this more general result is achieved in two steps. The first step is to show that
the deflation functor associated to $\pi$ (cf. \eqref{eq:unipro})
maps Hilbert$^{90}$ $R$-lattice functors
to projective functors of the gentle $R$-order category. 
The second step is to establish injectivity and surjectivity criteria
which ensure that a given natural transformation $\phi\colon\boX\to\boY$
envolving Hilbert$^{90}$ $R$-lattice functors is indeed an isomorphism
(cf. Prop.~\ref{prop:inj}, Prop.~\ref{prop:surmac}).

The first step is based on a sufficient criterion (cf. Thm.~\ref{thm:whitehead})
which guarantees that the deflation functor associated to a unitary projection $\pi$
is mapping $\circledast$-acyclic $R$-lattice functors to projective
$R$-lattice functors. Here $\circledast$ denotes the {\it Yoneda dual}
(cf. \S\ref{ss:yondual}) which can be seen as the standard dualizing procedure
for $R^\circledast$-categories.
Although this criterion is based on what is usually called
``abstract nonsense'', it will turn out to be quite useful:
two of the three hypothesis one has to claim
can be verified easily for the unitary projection $\pi$ and involve the
Hilbert$^{90}$ property, while the third is a direct consequence of Theorem D.

The two main results known to authors concerning the cohomology of cohomological
Mackey functors are due to S. Bouc (cf. \cite{bouc:com}) and D.~Tambara (cf. \cite{tamb:hom}),
but concern cohomological Mackey functors with values in a field of positive characteristic.
Although the just-mentioned results indicate that for cyclic groups
the theory of cohomological Mackey functors should be 
significantly easier (and different) than in the general case, the following 
consequence is nevertheless surprising.

\begin{thmE}
Let $R$ be a discrete valuation domain of characteristic $0$ 
with maximal ideal $pR$ for some prime number $p$, and 
let $G$ be a non-trivial finite cyclic $p$-group.
Then $\gldim_R(\mf_R(G))=3$.
\end{thmE}


\section{$R^\circledast$-categories}
\label{s:rord}
Let $R$ be a commutative ring with $1$, and let
$\RMod$ denote the abelian category of $R$-modules.
An $R$-module $M$ will be called an {\it $R$-lattice}, if $M$
is a finitely generated projective $R$-module.
We denote by $\Rlat$ the full subcategory of $\RMod$
the objects of which are $R$-lattices, and by $\RMod^{\fg}$
the full subcategory of $\RMod$ the objects of which are finitely
generated $R$-modules.
For certain applications we have to restrict our considerations to
Dedekind domains.
For such a ring $R$ one has the following property:
If $\phi\colon M\to Q$ is a surjective homomorphism of $R$-lattices, 
then $\ker(\phi)$ is an $R$-lattice
and the canonical map $\ker(\phi)\to M$ is split-injective.

Following \cite[Chap.~2. \S2]{bass:kth} one calls a category $\caC$ an {\it $R$-category},
if $\Hom_\caC(A,B)$ is an $R$-module for any pair of
objects $A,B\in\ob(\caC)$, and composition
\begin{equation}
\label{eq:bilin}
\argu\circ\argu\colon\Hom_\caC(B,C)\times\Hom_\caC(A,B)\longrightarrow\Hom_\caC(A,C)
\end{equation}
is $R$-bilinear for any three objects $A,B,C\in\ob(\caC)$. 
E.g., $\RMod$ is an $R$-category.
Note that $\caC^{\op}$ is an $R$-category for every $R$-category $\caC$.
A (covariant) functor $\phi\colon \caC\to\caD$ between $R$-categories $\caC$ and $\caD$
is called {\it $R$-linear}, if 
\begin{equation}
\label{eq:Rlin}
\phi_{A,B}\colon\Hom_{\caC}(A,B)\longrightarrow\Hom_{\caD}(\phi(A),\phi(B))
\end{equation}
is a homomorphism of $R$-modules for every pair of objects $A,B\in\ob(\caC)$.

\subsection{$R^\circledast$-order categories}
\label{ss:Rcat}
An $R$-category $\caC$ will be called an {\it $R$-order category}, if
$\ob(\caC)$ is a finite set and $\Hom_{\caC}(A,B)$ is an $R$-lattice
for all $A,B\in \ob(\caC)$.
E.g., if $\mu$ is an $R$-order, then $\mu\bullet$,
the category with one object $\bullet$ and 
$\Hom_{\mu\bullet}(\bullet,\bullet)=\mu$, is an $R$-order category.
An $R$-category $\caC$ together with an $R$-linear functor 
$\sigma\colon\caC\to\caC^{\op}$ satisfying
$\sigma(A)=A$ for all $A\in\ob(\caC)$ and $\sigma\circ\sigma=\iid_{\caC}$
will be called an {\it $R^\circledast$-category}. E.g., if $\mu$ is an $R$-algebra
with an $R$-linear antipode $\sigma_\mu\colon \mu\to\mu^{\op}$ of order $2$,
i.e., $\sigma\circ\sigma=\iid_{\mu}$,
then $\mu\bullet$ is an $R^\circledast$-category.
An $R^\circledast$-category $(\caC,\sigma)$, where $\caC$ is an $R$-order category,
will be called an {\it $R^\circledast$-order category}.

\subsection{Additive functors}
\label{ss:add}
Let $\caC$ be an $R$-category.
By $\euF_R(\caC^{\op},\RMod)$ we denote the category of
$R$-linear functors
from $\caC^{\op}$ to $\RMod$, i.e., $\boF\in\ob(\euF_R(\caC^{\op},\RMod))$
is a contravariant $R$-linear functor from $\caC$ to $\RMod$.
Morphisms in $\euF_R(\caC^{\op},\RMod)$ are given by the
{\it $R$-linear natural transformations}, i.e., $\eta\in\nat_R(\boF,\boG)$
is called {\it $R$-linear},
if $\eta_A\colon\boF(A)\to\boG(A)$ is $R$-linear for every $A\in\ob(\caC)$.
It is well known that $\euF_R(\caC^{\op},\RMod)$ is an abelian category
(cf. \cite[Chap.~IX, Prop.~3.1]{mcl:hom}).

A functor $\boF\in \ob(\euF_R(\caC^{\op},\RMod))$ 
will be called an {\it $R$-lattice functor} if $\boF(A)$ is an $R$-lattice 
for every object $A\in\ob(\caC)$.
By $\euF_R(\caC^{\op},\Rlat)\subseteq\euF_R(\caC^{\op},\RMod)$ 
we denote the full subcategory of $R$-lattice functors. 

Let $(\caC,\sigma)$ be an $R^\circledast$-category, and let
$\argu^\ast=\Hom_R(\argu,R)\colon\Rlat\longrightarrow\Rlat^{\op}$
denote the dualizing functor in $\Rlat$. Composition of $\argu^\ast$  with $\sigma$
yields a dualizing functor
\begin{equation}
\label{eq:equiv}
\argu^\ast\colon
\euF_R(\caC^{\op},\Rlat)\longrightarrow\euF_R(\caC^{\op},\Rlat)^{\op},
\end{equation}
where $\boF^\ast(A)=\boF(A)^\ast$ and $\boF^\ast(\phi)=\boF(\sigma(\phi))^\ast$ for
$\boF\in\ob(\euF_R(\caC^{\op},\Rlat))$ and $\phi\colon A\to B\in\Hom_\caC(A,B)$.

\subsection{Projectives}
\label{ss:proj}
Let $\caC$ be an $R$-category, and let $A\in\ob(\caC)$. Then 
\begin{equation}
\label{eq:defP}
\boP^A=\Hom_\caC(\argu,A)\in\ob(\euF_R(\caC^{\op},\RMod))
\end{equation}
is an $R$-linear functor from $\caC^{\op}$ to $\RMod$.
Moreover, if $\caC$ is an $R$-order category, then
$\boP^A$ is an $R$-lattice functor. One has
the following property (cf. \cite[Prop.~IV.7.3]{sten:roq}).

\begin{fact}
\label{fact:Pfirst}
Let $\caC$ be an $R$-category,
let $A\in\ob(\caC)$ and $\boF\in\ob(\euF_R(\caC^{\op},\RMod))$. Then 
one has
a canonical isomorphism
\begin{equation}
\label{eq:nattrans}
\theta_{A,\boF}\colon\nat_R(\boP^A,\boF)\longrightarrow\boF(A)
\end{equation}
given by $\theta_{A,\boF}(\xi)=\xi_A(\iid_A)$, $\xi\in\nat_R(\boP^A,\boF)$.
\end{fact}

The inverse of $\theta_{A,\boF}$ can be given explicit. For $f\in\boF(A)$ and
$B\in\ob(\caC)$ one has
\begin{equation}
\label{eq:yonnat}
\theta_{A,\boF}^{-1}(f)_B\colon\boP^A(B)\to\boF(B),\quad
\theta_{A,\boF}^{-1}(f)_B(\phi)=\boF(\phi)(f),\quad
\phi\in\Hom_\caC(B,A).
\end{equation}
It is straightforward to verify that $\theta_{A,\boF}^{-1}(f)\in\nat_R(\boP^A,\boF)$.
Let $\chi\colon\boF\to\boG$ be an $R$-linear natural transformation.
Then one has a commutative diagram
\begin{equation}
\label{eq:yonnat2}
\xymatrix{
\boF(A)\ar[d]_{\chi_A}\ar[r]^-{\theta_{A,\boF}^{-1}}&\nat_R(\boP^A,\boF)\ar[d]^{\chi\circ-}\\
\boG(A)\ar[r]^-{\theta_{A,\boG}^{-1}}&\nat_R(\boP^A,\boG).
}
\end{equation}
From this fact one concludes the following well known property (see \cite[Cor.~7.5]{sten:roq}).

\begin{fact}
\label{fact:yon}
Let $\caC$ be an $R$-category,
let $\boF,\boG\in\ob(\euF_R(\caC^{\op},\RMod))$, and let $\phi\in\nat_R(\boF,\boG)$
be a natural transformation such that $\phi_A\colon \boF(A)\to\boG(A)$ is surjective. 
Then for every natural transformation $\chi\in\nat_R(\boP^A,\boG)$ there exists  
$\tchi\in\nat_R(\boP^A,\boF)$ making the diagram
\begin{equation}
\label{eq:diayon}
\xymatrix{
&\boP^A\ar[d]^{\chi}\ar@{-->}[dl]_{\tchi}&\\
\boF\ar[r]^{\phi}&\boG\ar[r]&0
}
\end{equation}
commute. In particular, $\boP^A\in\ob(\euF_R(\caC^{\op},\RMod))$ is projective.
\end{fact}

As a consequence one has the following.

\begin{fact}
\label{fact:enproj}
Let $\caC$ be an $R$-category such that $\ob(\caC)$ is a set.
Then $\euF_R(\caC^{\op},\RMod)$ is an abelian category with enough
projectives.
\end{fact}

If $\caC$ is an $R$-category and $\ob(\caC)$ is a set,
we denote for $\boF\in\ob(\euF_R(\caC^{\op},\RMod))$
the right derived functors of
$\nat_R(\argu,\boF)$ by $\ext_R^k(\argu,\boF)$, $k\geq 0$.

Let $\caC$ be an $R$-order category. Then by definition
$\boP^A$ is an $R$-lattice functor.
A projective object $\boP\in\ob(\euF(\caC^{\op},\RMod))$
which is a lattice functor will be a called a {\it projective 
$R$-lattice functor}. For these functors one concludes the following:

\begin{fact}
\label{fact:rlatproj}
Let $R$ be a Dedekind domain, and let
$\caC$ be an $R$-order category.
Then every $R$-lattice functor $\boF\in\ob(\euF(\caC^{\op},\RMod))$
has a projective resolution $(\boP_k,\partial_k^\boP,\eps_\boF)$,
where $\boP_k$ is a projective $R$-lattice functor for every $k\geq 0$.
\end{fact}

\begin{rem}
\label{rem:rep} 
Let $\caC$ be an $R$-order category such that 
for all $A,B\in\ob(\caC)$, $A\not= B$, one has $A\not\simeq B$.
Let $\mu_\caC$ be the $R$-order given by
$\textstyle{\mu_\caC=\bigoplus_{A,B\in\ob(\caC)}\Hom_\caC(A,B)}$,
where the product is given by 
\begin{equation}
\label{eq:defmuC}
\alpha\cdot\beta=
\begin{cases}
\alpha\circ\beta&\ \text{for $B_1=B_2$,}\\
\hfil 0\hfil &\ \text{for $B_1\not=B_2$.}
\end{cases}
\end{equation}
for $\alpha\in\Hom_\caC(B_2,C)$, $\beta\in\Hom_\caC(A,B_1)$.
Then one has a canonical $R$-linear functor $\rho_\caC\colon\caC\to\mu_\caC\bullet$ (cf. \S\ref{ss:Rcat})
induced by the identity on morphisms.
Moreover, the category $\euF_R(\caC^{\op},\RMod)$ is naturally equivalent
to the category of right $\mu_{\caC}$-modules $\Mod_{\mu_{\caC}}$.
This equivalence is achieved by assigning a right $\mu_\caC$-module
$M$ the functor $\boF_M\in \ob(\euF_R(\caC^{\op},\RMod))$ given by
$\boF_M(A)=M\cdot\iid_A$ for $A\in\ob(\caC)$.
For $\phi\in\Hom_\caC(A,B)$ the mapping 
 $\boF_M(\phi)\colon\boF_M(B)\to\boF_M(A)$
is given by right multiplication with $\phi$.
A functor $\boF\in\ob(\euF_R(\caC^{\op},\RMod))$ can be made into a
right $\mu_\caC$-module $M_\boF$, where $M_\boF=\bigoplus_{A\in\ob(\caC)}\boF(A)$.
For $f\in\boF(B)$ and $\phi\in\Hom_{\caC}(A,B)$ one has
$f\cdot\phi=\boF(\phi)(f)$.

For $A\in\ob(\caC)$, $\iid_A$ is an idempotent in $\mu_{\caC}$.
Moreover, under the identification mentioned above $\boP^A$ corresponds
to the right $\mu_{\caC}$-module $\iid_A\cdot\mu_{\caC}$.
\end{rem}

\subsection{Dimensions}
\label{ss:Ldim}
Let $R$ be a Dedekind domain, let
$\caC$ be an $R$-order category, and let $\boF\in\ob(\euF_R(\caC^{\op},\RMod))$. 
Then $\boF$ has {\it projective $R$-dimension less or equal to $d$} if it has a
projective resolution $(\boP_k,\partial_k^\boP,\eps_\boF)$
with $\boP_k=0$ for $k>d$.
The minimal such number $d\in\N_0\cup\{\infty\}$ 
is called the {\it projection $R$-dimension} of $\boF$ and will be denoted
by $\prdim(\boF)$. The numbers
\begin{equation}
\label{eq:gldim}
\begin{aligned}
\gldim_R(\caC)&=\sup\{\,\prdim_R(\boF)\mid \boF\in\ob(\euF_R(\caC^{\op},\RMod))\,\},\\
\glatdim_R(\caC)&=\sup\{\,\prdim_R(\boF)\mid \boF\in\ob(\euF_R(\caC^{\op},\Rlat))\,\},
\end{aligned}
\end{equation}
will be called the {\it global $R$-dimension} and the {\it global $R$-lattice dimension}
of $\caC$, respectively. 
By a result of M.~Auslander, one has
\begin{equation}
\label{eq:gldim2}
\gldim_R(\caC)=\sup\{\,\prdim_R(\boF)\mid \boF\in\ob(\euF_R(\caC^{\op},\RMod^{\fg}))\,\}
\end{equation}
(cf. \cite[Thm.~9.12]{rotman:hom}).
In particular,
\begin{equation}
\label{eq:dims}
\glatdim_R(\caC)\leq\gldim_R(\caC)\leq\glatdim_R(\caC)+1.
\end{equation}
E.g., $\glatdim_R(\caC)=0$ if, and only if,
every $R$-lattice functor is projective.
An $R$-order category satisfying $\glatdim_R(\caC)\leq 1$ will be called {\it pseudo-hereditary}.
Such a category has the following property: Any subfunctor $\boF$ of a projective
$R$-lattice functor $\boP$ such that $\boP/\boF$ is an $R$-lattice functor is projective.

\subsection{The Yoneda dual}
\label{ss:yondual}
Let $(\caC,\sigma)$ be an $R^\circledast$-category.
For $\phi\in\Hom_{\caC}(A,B)$ one has 
an $R$-linear natural transformation $\boP(\phi)\colon\boP^A\to\boP^B$ 
given by composition
with $\phi$. 
Hence one has a functor
\begin{equation}
\label{eq:cdash}
\argu^\circledast\colon
\euF_R(\caC^{\op},\RMod)\longrightarrow \euF_R(\caC^{\op},\RMod)^{\op},
\end{equation}
given by $\boF^\circledast(A)=\nat_R(\boF,\boP^A)$
for $\boF\in\ob(\euF_R(\caC^{\op},\RMod))$ and $A\in\ob(\caC)$,
and $\boF^\circledast(\phi)=\boP(\sigma(\phi))\circ\argu\colon\boF^\circledast(B)\to\boF^\circledast(A)$
for  $\phi\colon A\to B\in\Hom_{\caC}(A,B)$.
We call the functor $\argu^\circledast$
the {\it Yoneda dual}. 

\begin{rem}
\label{rem:yondual}
Let $\mu$ be an $R$-algebra with $R$-linear antipode $\sigma\colon\mu\to\mu^{\op}$.
Then $\euF_R(\mu\bullet^{\op},\RMod)$ can be identified with the category of right $\mu$-modules
(cf. Rem.~\ref{rem:rep}).
Under this identification, the Yoneda dual satisfies 
$\argu^\circledast=\Hom_\mu(\argu,\mu)^\times$. Here we used the symbol ${}^\times$
to express that for a right $\mu$-module $M$, the left $\mu$-module $\Hom_\mu(M,\mu)$
is considered as right $\mu$-module via the map $\sigma$.
\end{rem}

The Yoneda dual has the following property:

\begin{prop}
\label{prop:yonproj}
Let $(\caC,\sigma)$ be an $R^\circledast$-category, and let $A\in\ob(\caC)$.
Then one has a canonical natural isomorphism
$j_A\colon (\boP^A)^\circledast\to\boP^A$
which is natural in $A$, i.e., for $\psi\colon A\to D\in\Hom_\caC(A,D)$ one has 
a commutative diagram
\begin{equation}
\label{eq:isoyon0}
\xymatrix{
(\boP^D)^\circledast\ar[d]_{\boP(\psi)^\circledast}\ar[r]^{j_D}&
\boP^D\ar[d]^{\boP(\sigma(\psi))}\\
(\boP^A)^\circledast\ar[r]^{j_A}&\boP^A.
}
\end{equation}
In particular, if $R$ is a Dedekind domain and $(\caC,\sigma)$ is an $R^\circledast$-order category,
then $\argu^\circledast$ maps projective $R$-lattice functors
to projective $R$-lattice functors, and $R$-lattice functors to $R$-lattice functors.
\end{prop}

\begin{proof} Let $\phi\colon B\to C$ be a morphism in $\caC$. By the definition
of $\boP^{\argu}$ and Fact~\ref{fact:Pfirst}, one has
canonical isomorphisms
\begin{equation}
\label{eq:isoyon1}
\xymatrix@R=.3truecm{
(\boP^A)^\circledast(B)\ar@{=}[d]\ar[0,3]^{j_A(B)}&&&\boP^A(B)\ar@{=}[d]\\
\nat_R(\boP^A,\boP^B)\ar[r]^-{\theta_{A,\boP^B}}&\boP^B(A)\ar@{=}[r]&\Hom_\caC(A,B)\ar[r]^{\sigma}&\Hom_{\caC}(B,A).}
\end{equation}
and the diagram
\begin{equation}
\label{eq:isoyon2}
\xymatrix{
\nat_R(\boP^A,\boP^C)\ar[r]\ar[d]_{\boP(\sigma(\phi))\circ-}
&\Hom_\caC(A,C)\ar[r]^{\sigma}\ar[d]^{\sigma(\phi)\circ-}
&\Hom_{\caC}(C,A)\ar[d]^{-\circ\phi}\\
\nat_R(\boP^A,\boP^B)\ar[r]&\Hom_\caC(A,B)\ar[r]^{\sigma}&\Hom_{\caC}(B,A)\\
}
\end{equation}
commutes. This shows that $j_A$ is a natural isomorphism.
The commutativity of the diagram
\begin{equation}
\label{eq:isoyon3}
\xymatrix{
\nat_R(\boP^D,\boP^B)\ar[r]\ar[d]_{-\circ\boP(\psi)}
&\Hom_\caC(D,B)\ar[r]^{\sigma}\ar[d]^{-\circ\psi}
&\Hom_{\caC}(B,D)\ar[d]^{\sigma(\psi)\circ-}\\
\nat_R(\boP^A,\boP^B)\ar[r]&\Hom_\caC(A,B)\ar[r]^{\sigma}&\Hom_{\caC}(B,A)\\
}
\end{equation}
shows the commutativity of \eqref{eq:isoyon0}. 
The final remark is straightforward.
\end{proof}

Let $A,B\in\ob(\caC)$ and let $\boF\in\ob(\euF_R(\caC^{\op},\RMod))$. For 
any $\chi\in\nat_R(\boF,\boP^B)$ one has an $R$-linear map
\begin{equation}
\label{eq:eta0}
\sigma\circ\chi_A\colon\xymatrix{
\boF(A)\ar[r]^{\chi_A}&\boP^B(A)\ar[r]^{\sigma_{A,B}}&\boP^A(B)}.
\end{equation}
Let $f\in\boF(A)$, and let $\eta^{f,B}_{\boF,A}\colon\nat_R(\boF,\boP^B)\longrightarrow\boP^A(B)$ be given by
\begin{equation}
\label{eq:eta1}
\eta^{f,B}_{\boF,A}(\chi)=\sigma(\chi_A(f)).
\end{equation}
For $\phi\colon B\to C\in\Hom_\caC(B,C)$ one has a commutative diagram
\begin{equation}
\label{eq:eta2}
\xymatrix{
\nat_R(\boF,\boP^C)\ar[r]^-{\eta^{f,C}_{\boF,A}}\ar[d]_{\boP(\sigma(\phi))\circ-}
&\boP^A(C)\ar[d]^{-\circ\phi}\\
\nat_R(\boF,\boP^B)\ar[r]^-{\eta^{f,B}_{\boF,A}}&\boP^A(B).\\
}
\end{equation}
Hence $\eta^{f,-}_{\boF,A}\colon\boF^\circledast\to\boP^A$ is an $R$-linear natural transformation. 
The mapping
\begin{equation}
\label{eq:eta3}
\eta_{\boF,A}\colon \boF(A)\longrightarrow\nat_R(\boF^\circledast,\boP^A)
\end{equation}
is $R$-linear, and for $\psi\colon D\to A\in\Hom_\caC(D,A)$, the diagram
\begin{equation}
\label{eq:eta4}
\xymatrix{
\boF(A)\ar[d]_{\boF(\psi)}\ar[r]^-{\eta_{\boF,A}}&
\nat_R(\boF^\circledast,\boP^A)\ar[d]^{\boP(\sigma(\psi))\circ-}\\
\boF(D)\ar[r]^-{\eta_{\boF,D}}&\nat_R(\boF^\circledast,\boP^D)
}
\end{equation}
commutes. Thus it defines an $R$-linear, natural transformation 
$\eta_\boF\colon\boF\to\boF^{\circledast\circledast}$.
Let $\alpha\in\nat_R(\boF,\boG)$. For all $A\in\ob(\caC)$ one has a commutative
diagram
\begin{equation}
\label{eq:eta5}
\xymatrix{
\boF(A)\ar[d]_{\alpha_A}
\ar[r]^-{\eta_{\boF,A}}&\nat_R(\boF^\circledast,\boP^A)\ar[d]^{-\circ\alpha^\circledast}\\
\boG(A)\ar[r]^-{\eta_{\boG,A}}&\nat_R(\boG^\circledast,\boP^A).
}
\end{equation}
Hence one has the following.

\begin{prop}
\label{prop:isoeta}
Let $(\caC,\sigma)$ be an $R^\circledast$-category. Then 
\begin{equation}
\label{eq:nateta}
\eta\colon\iid_{\euF_R(\caC^{\op},\RMod)}\longrightarrow\argu^{\circledast\circledast}
\end{equation}
is a natural transformation.
For every $E\in\ob(\caC)$, $\eta_{\boP^E}\colon\boP^E\to(\boP^E)^{\circledast\circledast}$
is a natural isomorphism. In particular, if $R$ is a Dedekind domain and
$(\caC,\sigma)$ is an $R^\circledast$-order category,
then $\eta_{\boP}\colon \boP\to\boP^{\circledast\circledast}$
is an isomorphism for every projective $R$-lattice functor $\boP\in\ob(\euF_R(\caC^{\op},\RMod))$.
\end{prop}

\begin{proof}
It suffices to show that $\eta_{\boP^E}\colon\boP^E\to(\boP^E)^{\circledast\circledast}$
is a natural isomorphism for every $E\in\ob(\caC)$.
For $A\in\ob(\caC)$ one has a commutative diagram
\begin{equation}
\label{eq:eta6}
\xymatrix@C=2.0truecm{
\boP^A(E)\ar[r]^-{\theta^{-1}_{E,\boP^A}}\ar[d]_{\sigma}&\nat_R(\boP^E,\boP^A)\ar[d]^{-\circ j_E}\\
\boP^E(A)\ar[r]^-{\eta_{\boP^E,A}}&\nat_R((\boP^E)^\circledast,\boP^A),
}
\end{equation}
and all maps apart from $\eta_{\boP^E,A}$ are isomorphisms (cf. \eqref{eq:yonnat}, \eqref{eq:isoyon1}).
Hence $\eta_{\boP^E,A}$ is an isomorphism, and this yields the claim.
\end{proof}

\subsection{Derived functors of the Yoneda dual}
\label{ss:deryon}
Let $(\caC,\sigma)$ be an $R^\circledast$-category such that $\ob(\caC)$ is a set.
Then $\euF_R(\caC^{\op},\RMod)$ is an abelian category with enough
projectives (cf. Fact~\ref{fact:enproj}).

The Yoneda dual $\argu^\circledast\colon \euF_R(\caC^{\op},\RMod)\to
\euF_R(\caC^{\op},\RMod)^{\op}$ is additive and left-exact.
Let $\caR^k(\argu)^\circledast$, $k\geq 1$, denote its right-derived functors, i.e.,
one has that
\begin{equation}
\label{eq:rightyo}
\caR^k(\boF)^\circledast(A)=\ext_R^k(\boF,\boP^A),\qquad
\text{for $\boF\in\ob(\euF_R(\caC^{\op},\RMod))$,}
\end{equation}
and $\caR^k(\boF)^\circledast(\phi)=\ext_R^k(\boF,\boP(\sigma(\phi)))$
for $\phi\in\Hom_\caC(A,B)$.
A functor $\boF$ will be called {\it $\circledast$-acyclic},
if $\caR^k(\boF)^\circledast=0$ for all $k>0$. E.g., every projective functor 
is $\circledast$-acyclic.

Let $R$ be a Dedekind domain, and let $(\caC,\sigma)$ be an
$R^\circledast$-order category.
An $R$-lattice functor $\boF\in\ob(\euF_R(\caC^{\op},\RMod))$
will be called {\it $\circledast$-bi-acyclic}, if $\boF$ and
$\boF^\circledast$ are $\circledast$-acyclic.
The $R^\circledast$-order category $(\caC,\sigma)$ will be called {\it $\circledast$-symmetric},
if every $\circledast$-acyclic $R$-lattice functor is $\circledast$-bi-acyclic.

\subsection{Gorenstein projective functors}
\label{ss:compproj}
Let $R$ be a Dedekind domain, let $(\caC,\sigma)$ be an
$R^\circledast$-order category, and let 
$\boF\in\ob(\euF_R(\caC^{\op},\Rlat))$.
A chain complex $(\boP_\bullet,\partial_\bullet^\boP)$ together with a
natural transformation $\eps\colon \boP_0\to\boF$ will be called a
{\it complete projective $R$-lattice functor resolution} of $\boF$, if
\begin{itemize}
\item[(i)] $\boP_k$ is a projective $R$-lattice functor for all $k\in\Z$;
\item[(ii)] $(\boP_\bullet,\partial_\bullet^\boP)$ is exact;
\item[(iii)] $\eps\circ\partial_1=0$ and $\eps$ induces a natural isomorphism $\tilde{\eps}\colon\coker(\partial_1)\to\boF$.
\end{itemize}
For an exact chain complex of projective $R$-lattice functors $(\boP_\bullet,\partial_\bullet^\boP)$ we
denote by $(\boQ_\bullet,\partial_\bullet^\boQ)=(\boP_\bullet,\partial_\bullet^\boP)^\circledast$ the chain complex of
projective $R$-lattice functors given by
$\boQ_k=\boP^\circledast_{-k-1}$ and $\partial_k^\boQ=(\partial_{-k}^\boP)^\circledast$.
Note that by Proposition~\ref{prop:isoeta}, $(\boP_\bullet,\partial_\bullet^\boP)^{\circledast\circledast}$
is canonically isomorphic to $(\boP_\bullet,\partial_\bullet^\boP)$.
The complete projective $R$-lattice functor resolution $(\boP_\bullet,\partial_\bullet^\boP,\eps_\boF)$
of $\boF$ will be called {\it $\circledast$-exact} if
\begin{itemize}
\item[(iv)] $(\boP_\bullet,\partial_\bullet^\boP)^\circledast$ is exact.
\end{itemize}
An $R$-lattice functor with a $\circledast$-exact complete
projective $R$-lattice functor resolution is also called
a {\it Gorenstein projective functor}.
One has the following property.

\begin{prop}
\label{prop:explores}
Let $R$ be a Dedekind domain, let $(\caC,\sigma)$ be an
$R^\circledast$-order category, and let 
$\boF\in\ob(\euF_R(\caC^{\op},\Rlat))$ be an $R$-lattice functor.
Then $\boF$ is Gorenstein projective if, and only if,
$\boF$ is $\circledast$-bi-acyclic.
\end{prop}

\begin{proof}
Suppose that $\boF$ is $\circledast$-bi-acyclic.
By Fact~\ref{fact:rlatproj}, $\boF$ has a projective resolution 
$(\boP_\bullet,\partial_\bullet^\boP,\eps_\boF)$
by projective $R$-lattice functors. Let $\boQ_{k}=\boP^\circledast_{-k}$,
$\partial^\boQ_{k}=(\partial_{1-k}^\boP)^\circledast$. Then $(\boQ_\bullet,\partial^\boQ_\bullet)$
is a chain complex of projective $R$-lattice functors concentrated in non-positive degrees
(cf. Prop.~\ref{prop:yonproj}). As $\boF$ is $\circledast$-acyclic, one has
\begin{equation}
\label{eq:comres1}
H_k(\boQ_\bullet,\partial^\boQ_\bullet)\simeq
\begin{cases}
\boF^\circledast&\ \text{for $k=0$;}\\
\hfil 0\hfil&\ \text{for $k\not=0$.}
\end{cases}
\end{equation}
As $\boF^\circledast$ is an $R$-lattice functor, it has
a  projective resolution 
$(\boR_\bullet,\partial_\bullet^\boR,\mu_{\boF^\circledast})$
by projective $R$-lattice functors. Let $(\boS_\bullet,\partial_\bullet^\boS)$
be the chain complex given by
$\boS_k=\boR_k$ for $k\geq 0$ and $\boS_k=\boQ_{k+1}$ for $k<0$, and mappings
$\partial_k^\boS=\partial_k^\boR$ for $k\geq 1$,
$\partial_k^\boS=\partial_{k+1}^\boQ$ for $k\leq -1$, and 
$\partial_0=\eps_{\boF}^\circledast\circ\mu_{\boF^\circledast}$.
Then $(\boS_\bullet,\partial_\bullet^\boS,\mu_{\boF^\circledast})$
is a complete projective $R$-lattice functor resolution of $\boF^\circledast$.

Let $(\boT_\bullet,\partial_\bullet^\boT)=(\boS_\bullet,\partial_\bullet^\boS)^\circledast$. 
Then one has
$\boT_0=\boP_0^{\circledast\circledast}$, 
$\rho=\eps\circ\eta_{\boP_0}^{-1}\colon\boT_0\to\boF$ (cf. Prop.~\ref{prop:isoeta})
is satisfying $\rho\circ\partial_1^\boT=0$, and the induced map 
$\tilde{\rho}\colon\coker(\partial_1^\boT)\to\boF$ is a natural isomorphism.
By construction and Proposition~\ref{prop:isoeta}, 
one has $H_k(\boT_\bullet,\partial_\bullet^\boT)=0$ for $k>0$.
As $\boF^\circledast$ is $\circledast$-acyclic, one has also
$H_k(\boT_\bullet,\partial_\bullet^\boT)=0$ for $k<-1$. 

Let $\boT^{<0}_\bullet$ and $\boT^{\geq0}_\bullet$ denote the truncated
chain complexes, respectively, and
consider the short exact sequence of chain complexes
$0\to\boT^{<0}_\bullet\to\boT_\bullet\to\boT^{\geq0}_\bullet\to 0$. 
By construction, the connecting homomorphism
$H_0(\delta)\colon H_0(\boT^{\geq0}_\bullet)\to H_{-1}(\boT^{<0}_\bullet)$
is an isomorphism. The long exact sequence in homology implies that
the chain complex $(\boT_\bullet,\partial_\bullet^\boT)$ has trivial homology,
and hence is exact.
Thus by Proposition~\ref{prop:yonproj}, $(\boT_\bullet,\partial_\bullet^\boT,\rho)$
is a complete projective $R$-lattice functor resolution of $\boF$.
As $(\boT_\bullet,\partial_\bullet^\boT)^\circledast$ is canonically isomorphic to
$(\boS_\bullet,\partial_\bullet^\boS)$, $(\boT_\bullet,\partial_\bullet^\boT,\rho)$
is also $\circledast$-exact.

Let $(\boT_\bullet,\partial_\bullet^\boT,\eps_{\boF})$ be a 
$\circledast$-exact complete projective $R$-lattice functor
resolution of $\boF$.
Then
\begin{equation}
\label{eq:yonres1}
\caR^k(\boF)^\circledast=H_{-k-1}((\boT_\bullet,\partial_\bullet^\boT)^\circledast)=0,\ \ k>0,
\end{equation}
i.e., $\boF$ is $\circledast$-acyclic. Replacing the chain complex $(\boT_\bullet,\partial_\bullet^\boT)$
by the chain complex $(\boT_\bullet,\partial_\bullet^\boT)^\circledast$ shows that
$\boF^\circledast$ is also $\circledast$-acyclic. 
\end{proof}

\subsection{Gorenstein $R^\circledast$-order categories}
\label{ss:gor}
Let $R$ be a Dedekind\footnote{For an arbitrary commutative ring $R$ with $1$ the kernel of a surjective
homomorphism $\phi\colon M\to Q$ of $R$-lattices is not necessarily an $R$-lattice.
This is the reason why we restrict all subsequent considerations to $R$-order categories
over a Dedekind domain $R$.} domain, and let
$(\caC,\sigma)$ be an $R^\circledast$-order category. For $A\in\ob(\caC)$
the functors
\begin{equation}
\label{eq:defJ}
\boJ^A=(\boP^A)^\ast\in\ob(\euF_R(\caC^{\op},\RMod))
\end{equation}
are $R$-lattice functors which are {\it relative injective} in $\euF_R(\caC^{\op},\Rlat)$
in the following sense: Let $\alpha\colon\boF\to\boG$ be a split-injective,
$R$-linear transformation of $R$-lattice functors, and let $\beta\colon \boF\to\boJ^A$
be any $R$-linear natural transformation. Then there exists an $R$-linear natural
transformation $\tbeta\colon\boG\to\boJ^A$ such that the diagram
\begin{equation}
\label{eq:relinj}
\xymatrix{
0\ar[r]&\boF\ar[r]^\alpha\ar[d]_{\beta}&\boG\ar@{-->}[dl]^{\tbeta}\\
&\boJ^A&
}
\end{equation}
commutes.
Here we called a natural transformation $\alpha\colon\boF\to\boG$ 
of $R$-lattice functors {\it split injective},
if it is injective and $\coker(\alpha_B)$ is an $R$-lattice for every $B\in\ob(\caC)$.
The $R^\circledast$-order category $(\caC,\sigma)$ is called {\it $m$-Gorenstein}, $m\geq 0$, if
$\prdim(\boJ^A)\leq m$ for all $A\in\ob(\caC)$.
A $0$-Gorenstein $R^\circledast$-order category is also called {\it Frobenius}.
For $m$-Gorenstein $R^\circledast$-order categories one has the following:

\begin{prop}
\label{prop:gor}
Let $R$ be a Dedekind domain, and let $(\caC,\sigma)$ be an $m$-Gorenstein 
$R^\circledast$-order category.
Then for any $\boF\in\ob(\euF_R(\caC^{\op},\Rlat))$ and $k> m$ one has
$\caR^k(\boF)^\circledast=0$.
\end{prop}

\begin{proof}
By Fact~\ref{fact:rlatproj}, $\boF$ has a projective resolution $(\boP_i,\partial_i^\boP,\eps_\boF)$
by projective $R$-lattice functors.
Moreover, by hypothesis, for $A\in\ob(\caC)$ the functor $\boJ^A$ has 
a finite projective resolution $(\boQ_j,\partial_j^\boQ,\eps_A)$ 
by projective $R$-lattice functors
and $\boQ_j=0$ for $j>m$.
Thus $\boP^A\simeq(\boJ^A)^\ast$ has a finite, relative injective resolution $(\boI^j,\delta^j,\mu_A)$,
$\boI^j=\boQ_j^\ast$, $\delta^{j-1}=(\partial_j^\boQ)^\ast$, $\mu_A=\eps_A^\ast$ with $\boI^j=0$ for
$j>m$. Consider the double complex $(E_0^{s,t},\partial_v,\partial_h)$,
where $E_0^{s,t}=\nat_R(\boP_s,\boI^t)$ and $\partial_v$ and $\partial_h$ 
are the vertical and horizontal differential induced by $\partial_\bullet^\boP$ and $\delta^\bullet$,
respectively. The cohomology of the total complex $(\Tot^\bullet(E_0^{s,t}),\partial_v+(-1)^\bullet\partial_h)$
can be calculated in two ways.

Applying first the vertical and then the horizontal differential yields a spectral sequence
with $E_2$-term
\begin{equation}
\label{eq:Ev}
{}^vE_2^{s,t}=
\begin{cases}
\ext_R^s(\boF,\boP^A)&\ \ \text{for $t=0$},\\
\hfil 0\hfil&\ \ \text{for $t\not=0$,}
\end{cases}
\end{equation}
concentrated on the $(t=0)$-line.
By definition, $\nat_R(\argu,\boI^j)$ is exact for every short exact sequence of 
$R$-lattice functors. Since $R$ is a Dedekind domain, 
$0\to\kernel(\partial_s^\boP)\to\boP_s\to\image(\partial_s^\boP)\to 0$ is a short exact
sequence of $R$-lattice functors for every $s\geq 0$.
Hence applying first the horizontal and then the vertical differential yields a spectral sequence
with $E_1$-term concentrated on the $(s=0)$-line, and 
${}^hE_1^{0,t}=0$ for $t>m$.
The claim then follows from the fact that both spectral sequences converge to the cohomology of the total complex.
\end{proof}

\subsection{$R^\circledast$-order categories with the Whitehead property}
\label{ss:beauty}
Let $R$ be a Dedekind domain.
The Gorenstein property of an $R^\circledast$-order category 
is a quantitative measurement for the failure of being Frobenius.  
However, for our main purpose another property
plays a more important role. 
We say that an $R^\circledast$-order category $(\caC,\sigma)$ has the 
{\it Whitehead property}\footnote{The famous {\it Whitehead problem},
stated by J.~H.~Whitehead around 1950,
is the question whether every abelian group $A$ satisfying $\Ext^1_\Z(A,\Z)=0$ must be 
a free abelian group. For finitely generated abelian groups this is easily seen to be true,
and  K.~Stein showed (cf. \cite{stein:ab})
that the statement remains valid for countable abelian groups.
However, by the extra-ordinary work of S.~Shelah (cf. \cite{shel:wh1},
\cite{shel:wh2}, \cite{shel:wh3}) one knows now that 
this problem is in general undecidable.},
if any $\circledast$-acyclic $R$-lattice functor is projective.
The following property is well known (cf. \cite[Prop.~VIII.6.7]{brown:coh}).

\begin{fact}
\label{fact:glwh}
Let $R$ be a Dedekind domain, and let $(\caC,\sigma)$ be an
$R^\circledast$-order category of finite global $R$-lattice dimension.
Then $(\caC,\sigma)$ is Gorenstein and has the Whitehead property.
Moreover,
\begin{equation}
\label{eq:glwh}
\glatdim_R(\caC)=\max\{\,k\geq 0\mid\caR^k(\argu)^\circledast\not=0\,\}.
\end{equation}
\end{fact}

\begin{rem}
\label{rem:glwh}
Let $R$ be a Dedekind domain, and let $(\caC,\sigma)$ be an
$R^\circledast$-order category. For $m\geq 0$ one has the implications
\begin{equation}
\label{eq:glwh2}
\gldim_R(\caC)\leq m\ \Longrightarrow\ 
\caC\ \text{$m$-Gorenstein \& Whitehead}\ \Longrightarrow\ 
\caC\ \text{$m$-Gorenstein}.
\end{equation}
If $G$ is a finite group, then $(\Z[G]\bullet,\sigma)$, where
$\sigma(g)=g^{-1}$ for $g\in G$, is $0$-Gorenstein.
But $(\Z[G]\bullet,\sigma)$ has the Whitehead property 
if, and only if, $G$ is the trivial group.
Hence the second implication cannot be reversed.
For certain values of $m$ one can reverse
the first implication. E.g., if 
$(\caC,\sigma)$ is $0$-Gorenstein, then it has the Whitehead property if, and only if, 
every $R$-lattice functor is projective, i.e.,
$\glatdim_R(\caC)=0$.
This is also the case for $m=1$.
\end{rem}

\begin{fact}
\label{fact:gldim1}
Let $R$ be a Dedekind domain, and let $(\caC,\sigma)$ be 
a $1$-Gorenstein $R^\circledast$-order category.
Then $(\caC,\sigma)$ has the Whitehead property if, and only if, 
$\caC$ is pseudo-hereditary, i.e., $\glatdim_R(\caC)\leq 1$.
\end{fact}

\begin{proof} By Fact \ref{fact:glwh}, it suffice to show the
reverse direction of the first implication of \eqref{eq:glwh2}.
Suppose that $(\caC,\sigma)$ is $1$-Gorenstein and has the Whitehead property.
Let $\boF\in\ob(\euF_R(\caC^{\op},\Rlat))$.
Then there exists a surjective natural transformation $\pi\colon \boP\to\boF$
for some projective $R$-lattice functor $\boP$, and $\boQ=\kernel(\pi)$
is an $R$-lattice functor. By 
Proposition~\ref{prop:gor} and the long exact sequence, the sequence
\begin{equation}
\label{eq:semiher2}
\caR^1(\boF)^\circledast\longrightarrow
\caR^1(\boP)^\circledast\longrightarrow
\caR^1(\boQ)^\circledast\longrightarrow 0
\end{equation}
is exact.
As $\caR^1(\boP)^\circledast=0$, $\boQ$ is $\circledast$-acyclic
and thus, by hypothesis, projective.
\end{proof}

\subsection{Functors between $R^\circledast$-categories}
\label{ss:funRcat}
Let $(\caC,\sigma_\caC)$ and $(\caD,\sigma_\caD)$ be 
$R^\circledast$-categories.  An $R$-linear functor $\phi\colon\caC\to\caD$
will be called {\it unitary}, if
\begin{equation}
\label{eq:Runi}
\sigma_\caD\circ\phi=\phi\circ\sigma_{\caC}.
\end{equation}
If $\ob(\caC)$ and $\ob(\caD)$ are sets the unitary functor
$\pi\colon(\caC,\sigma_\caC)\to(\caD,\sigma_\caD)$ will be called a {\it unitary projection}, if
$\pi\colon\ob(\caC)\longrightarrow\ob(\caD)$ is a bijection, and
\begin{equation}
\label{eq:Runi2}
\pi_{A,B}\colon\Hom_\caC(A,B)\longrightarrow\Hom_\caD(\pi(A),\pi(B))
\end{equation}
is surjective for any pair of objects $A,B\in\ob(\caC)$.
For such a functor composition with $\pi$ induces
an exact inflation functor
\begin{equation}
\label{eq:inf}
\ifl^\pi(\argu)=\argu\circ\pi\colon\euF_R(\caD^{\op},\RMod)\longrightarrow\euF_R(\caC^{\op},\RMod).
\end{equation}
Let $\pi\colon (\caC,\sigma_\caC)\to(\caD,\sigma_\caD)$ be a unitary
projection of $R^\circledast$-order categories. Then $\pi$ induces a surjective
homomorphisms of $R$-orders $\mu(\pi)\colon\mu_{\caC}\to\mu_{\caD}$ (cf. Rem.~\ref{rem:rep}).
Moreover, the inflation functor $\ifl_{\mu_{\caD}}^{\mu_{\caC}}(\argu)\colon\Mod_{\mu_{\caD}}\to
\Mod_{\mu_{\caC}}$ has a left-adjoint
\begin{equation}
\label{eq:dfl1}
\dfl_{\mu_{\caD}}^{\mu_{\caC}}(\argu)=\argu\otimes_{\mu_{\caC}}\mu_{\caD}\colon\Mod_{\mu_{\caC}}\to
\Mod_{\mu_{\caD}}.
\end{equation}
From this fact one concludes the following.

\begin{fact}
\label{fact:dfl}
Let $\pi\colon (\caC,\sigma_\caC)\to(\caD,\sigma_\caD)$ be a unitary
projection of $R^\circledast$-order categories.
\begin{itemize}
\item[(a)] There exists a functor
\begin{equation}
\label{eq:dfl2}
\dfl^\pi(\argu)\colon \euF_R(\caC^{\op},\RMod)\longrightarrow\euF_R(\caD^{\op},\RMod)
\end{equation}
which is left-adjoint to $\ifl^\pi(\argu)$.
\item[(b)] The unit of the adjunction
$\eta\colon\iid_{\euF_R(\caC^{\op},\RMod)}\longrightarrow \ifl^\pi\circ\dfl^\pi$
is a natural surjection, and 
the the co-unit
$\varepsilon\colon \dfl^\pi\circ\ifl^\pi\longrightarrow\iid_{\euF_R(\caD^{\op},\RMod)}$
is a natural isomorphism.
\item[(c)] For all $A\in\ob(\caC)$ there exists an isomorphism
$\xi_A\colon \dfl^\pi(\boP^A)\to\boP^{\pi(A)}$ making the diagram
\begin{equation}
\label{eq:dlf3}
\xymatrix{
\dfl^\pi(\boP^A)\ar[r]^-{\xi_A}\ar[d]_{\dfl^\pi(\boP(\phi))}&\boP^{\pi(A)}\ar[d]^{\boP(\pi(\phi))}\\
\dfl^\pi(\boP^B)\ar[r]^-{\xi_B}&\boP^{\pi(B)}
}
\end{equation}
commute for all $\phi\colon A\to B\in\Hom_\caC(A,B)$.
\end{itemize}
\end{fact}

For $A\in\ob(\caC)$ and $\phi\in\Hom_\caC(A,B)$
put $\boQ^A=\ifl^\pi(\boP^{\pi(A)})$ and 
$\boQ(\phi)=\ifl^\pi(\boP(\pi(\phi)))$. Then
\begin{equation}
\label{eq:dfl4}
\xymatrix@C=2truecm{
\tau_A\colon \boP^A\ar[r]^-{\eta_{\boP^A}} &\ifl^\pi(\dfl^\pi(\boP^A))\ar[r]^-{\ifl^\pi(\xi_A)}&\boQ^A
}
\end{equation}
is a surjection satisfying 
$\tau_B\circ\boP(\phi)=\boQ(\phi)\circ\tau_A$
for all $\phi\in\Hom_\caC(A,B)$.

For $\boF\in\ob(\euF_R(\caC^{\op},\RMod))$ let $\boF^\boxtimes\in\ob(\euF_R(\caC^{\op},\RMod))$
be the functor given by
$\boF^\boxtimes(A)=\nat_R(\boF,\boQ^A)$ and 
$\boF^\boxtimes(\phi)=\boQ(\sigma_\caC(\phi))\circ\argu$
for $\phi\in\Hom_{\caC}(A,B)$.
Then $\argu^\boxtimes\colon\euF_R(\caC^{\op},\RMod)\to\euF_R(\caC^{\op},\RMod)^{\op}$
is a functor. By \eqref{eq:Runi}, one has 
\begin{equation}
\label{eq:dfl5}
\boP(\pi(\sigma_\caC(\phi)))=\boP(\sigma_\caD(\pi(\phi)))\colon \boP^{\pi(B)}\to\boP^{\pi(A)}
\end{equation}
for all $\phi\in\Hom_\caC(A,B)$. Hence the mapping 
$\ttau\colon\argu^\circledast\longrightarrow\argu^\boxtimes$ induced by $\tau$
is a natural transformation.
Since
\begin{equation}
\label{eq:dfl6}
\boF^\boxtimes(A)=\nat_R(\boF,\boQ^A)\simeq\nat_R(\dfl^\pi(\boF),\boP^{\pi(A)})=
\ifl^\pi(\dfl^\pi(\boF)^\circledast),
\end{equation}
$\argu^\boxtimes$ can be identified with $\ifl^\pi(\dfl^\pi(\argu)^\circledast)$.
Thus by the left-adjointness of $\dfl^\pi(\argu)$, $\ttau$ induces a natural transformation
\begin{equation}
\label{eq:dfl7}
\tttau\colon \dfl^\pi(\argu^\circledast)\longrightarrow\dfl^\pi(\argu)^\circledast
\colon\euF_R(\caC^{\op},\RMod)\longrightarrow\euF_R(\caD^{\op},\RMod)^{\op}.
\end{equation}
For $A\in\ob(\caC)$ the mapping 
$\tttau_{\boP^A}\colon \dfl^\pi((\boP^A)^\circledast)\to\dfl(\boP_A)^\circledast$
coincides with the isomorphism
\begin{equation}
\label{eq:dfl8}
\tttau_{\boP^A}\colon \Hom_{\mu_{\caC}}(\iid_A\cdot\mu_{\caC},\mu_{\caC})^\times
\otimes_{\mu_{\caC}}\mu_{\caD}\longrightarrow
\Hom_{\mu_{\caD}}(\iid_{\pi(A)}\cdot\mu_{\caD},\mu_{\caD})^\times.
\end{equation}
From this fact one concludes the following.

\begin{fact}
\label{fact:uni}
Let $\pi\colon (\caC,\sigma_\caC)\to(\caD,\sigma_\caD)$ be a unitary
projection of $R^\circledast$-order categories. Then
$\tttau_{\boP}\colon\dfl^\pi(\boP^\circledast)\to\dfl^\pi(\boP)^\circledast$
is an isomorphism for every projective $R$-lattice functor
$\boP\in\ob(\euF_R(\caC^{\op},\RMod))$.
\end{fact}

If $\pi\colon (\caC,\sigma_\caC)\to(\caD,\sigma_\caD)$ is a unitary
projection of $R^\ast$-order categories, its deflation functor
$\dfl^\pi(\argu)\colon\euF_R(\caC^{\op},\RMod)\to\euF_R(\caD^{\op},\RMod)$ 
is right exact and maps projectives to projectives 
(cf. \cite[Prop.~2.3.10]{weib:hom}). 
We denote by
$\caL_k\dfl^\pi(\argu)$ its left derived functors.
Functors $\boF\in\ob(\euF(\caC^{\op},\RMod))$ satisfying $\caL_k\dfl^\pi(\boF)=0$
for all $k>0$ will be called {\it $\pi$-acyclic}.

\begin{thm}
\label{thm:whitehead}
Let $R$ be a Dedekind domain, and let
$\pi\colon(\caC,\sigma_\caC)\to(\caD,\sigma_\caD)$
be a unitary projection of $R^\ast$-order categories.
Assume further that 
\begin{itemize}
\item[{\rm (i)}] $(\caC,\sigma_\caC)$ is $\circledast$-symmetric \textup{(cf. \S\ref{ss:deryon})};
\item[{\rm (ii)}] $(\caD,\sigma_\caD)$ has the Whitehead property \textup{(cf. \S\ref{ss:beauty})};
\item[{\rm (iii)}] if $\boF^\circledast\in\ob(\euF_R(\caC^{\op},\Rlat))$ is
$\circledast$-acyclic, $\boF$ is also $\pi$-acyclic.
\end{itemize}
Then $\dfl^\pi(\boG)$ is a projective $R$-lattice functor for any $\circledast$-acyclic
$R$-lattice functor $\boG\in\ob(\euF_R(\caC^{\op},\Rlat))$.
\end{thm}

\begin{proof}
Suppose that $\boG\in\ob(\euF_R(\caC^{\op},\Rlat))$ is $\circledast$-acyclic.
By hypothesis (i), $\boG$ is $\circledast$-bi-acyclic and thus Gorenstein projective
(cf. Prop.~\ref{prop:explores}),
i.e., $\boG$ admits a 
$\circledast$-exact complete projective $R$-lattice functor resolution
$(\boP_\bullet,\partial_\bullet^\boP,\eps_\boG)$.
Shifting the chain complex $(\boP_\bullet,\partial_\bullet^\boP)$ appropriately,
one concludes that every functor $\boC_k=\coker(\partial_{k+1})$ admits a 
$\circledast$-exact complete projective $R$-lattice functor resolution
for all $k\in\Z$, and hence is Gorenstein projective. 
Thus by Proposition~\ref{prop:explores}, $\boC_k$ is $\circledast$-bi-acyclic, i.e.,
$\boC_k$ and $\boC_k^\circledast$ are $\pi$-acyclic for all $k\in\Z$.

Let $(\boQ_\bullet,\partial_\bullet^\boQ)=(\dfl^\pi(\boP_\bullet),\dfl^\pi(\partial_\bullet^\boP))$.
As
$H_k(\boQ_\bullet,\partial_\bullet^\boQ)\simeq \caL_1\dfl^\pi(\boC_{k-1})=0$,
$(\boQ_\bullet,\partial_\bullet^\boQ)$ is exact. By Fact~\ref{fact:uni}, one has an isomorphism
of chain complexes
\begin{equation}
\label{eq:uni1}
(\boQ_\bullet,\partial_\bullet^\boQ)^\circledast\simeq (\dfl^\pi((\boP_\bullet,\partial_\bullet^\boP)^\circledast).
\end{equation}
Moreover, as $H_k((\boQ_\bullet,\partial_\bullet^\boQ)^\circledast)\simeq
\caL_1\dfl^\pi(\boC^\circledast_{1-k})=0$ (cf. \S\ref{ss:compproj}), 
$(\boQ_\bullet,\partial_\bullet^\boQ)^\circledast$ is also exact. Hence
$(\boQ_\bullet,\partial_\bullet^\boQ,\dfl^\pi(\eps_\boG))$ is a 
$\circledast$-exact complete projective $R$-lattice functor resolution of
$\dfl^\pi(\boG)$. In particular, $\dfl^\pi(\boG)$ is Gorenstein projective 
and thus $\circledast$-bi-acyclic
(cf. Prop.~\ref{prop:explores}). Hence by hypothesis (ii), $\dfl^\pi(\boG)$ is projective.
\end{proof}


\section{Cohomological Mackey functors}
\label{ss:mack}
Throughout this section $G$ will denote a finite group, 
and - if not stated otherwise -
$R$ will denote a commutative ring with unit $1_R\in R$.

\subsection{Cohomological $G$-Mackey functors}
\label{ss:Gmac1}
A {\it cohomological $G$-Mackey functor $\boX$}
with values in the category of $R$-modules
is a family of $R$-modules $(\boX_U)_{U\subseteq G}$ together with
homomorphisms of $R$-modules
\begin{equation}
\label{eq:mac1}
\begin{aligned}
i_{U,V}^\boX\colon&\boX_U\longrightarrow \boX_V,\\
t_{V,U}^\boX\colon&\boX_V\longrightarrow \boX_U,\\
c_{g,U}^\boX\colon& \boX_U\longrightarrow\boX_{{}^gU},
\end{aligned}
\end{equation}
for $U,V\subseteq G$, $V\subseteq U$, $g\in G$, which satisfy the identities:
\begin{itemize}
\item[(cMF$_1$)] $i_{U,U}^\boX=t_{U,U}^\boX=c_{u,U}^\boX=\iid_{\boX_U}$ for all $U\subseteq G$
and all $u\in U$;
\item[(cMF$_2$)] $i_{V,W}^\boX\circ i_{U,V}^\boX=i_{U,W}^\boX$ and
$t_{V,U}^\boX\circ t_{W,V}^\boX=t_{W,U}^\boX$ for all $U,V,W\subseteq G$ and 
$W\subseteq V\subseteq U$;
\item[(cMF$_3$)] $c_{h,{}^gU}^\boX\circ c_{g,U}^\boX=c_{hg,U}^\boX$ for all $U\subseteq G$ and 
$g,h\in G$;
\item[(cMF$_4$)] $i_{{}^gU,{}^gV}^\boX\circ c_{g,U}^\boX=c_{g,V}^\boX\circ i_{U,V}^\boX$ 
for all $U,V\subseteq G$, $V\subseteq U$, and $g\in G$;
\item[(cMF$_5$)] $t_{{}^gV,{}^gU}^\boX\circ c_{g,V}^\boX=c_{g,U}^\boX\circ t_{V,U}^\boX$ 
for all $U,V\subseteq G$, $V\subseteq U$, and $g\in G$;
\item[(cMF$_6$)] $i_{U,W}^\boX\circ t_{V,U}^\boX=\sum_{g\in W\setminus U/V}
t_{ {}^gV\cap W,W}^\boX\circ c_{g,V\cap W^g}^\boX\circ i_{V,V\cap W^g}^\boX$, where $W^g=g^{-1}Wg$
for all subgroups $U, V, W\subseteq G$ and 
$V,W\subseteq U$;
\item[(cMF$_7$)] $t_{V,U}^\boX\circ i_{U,V}^\boX=|U:V|.\iid_{\boX_U}$ 
for all subgroups $U,V\subseteq G$, $V\subseteq U$.
\end{itemize}
A homomorphism of cohomological Mackey functors $\phi\colon\boX\to\boY$ is a family
of $R$-module homomorphisms $\phi_U\colon\boX_U\to\boY_U$, $U\subseteq G$ which commute with all the mappings
$i_{.,.}$, $t_{.,.}$ and $c_{g,.}$, $g\in G$.
By $\MC_G(\RMod)$ we denote the abelian category of all cohomological $G$-Mackey functors
with values in the category of $R$-modules.  
For $\boX$, $\boY\in\ob(\MC_G(\RMod))$
we denote by $\nat_G(\boX,\boY)$ the morphisms in the category $\MC_G(\RMod)$.
For further details on Mackey functors see
\cite{dress:mac}, \cite{pj:simp}, \cite{pw:user}.

\subsection{The Mackey category}
\label{ss:maccat}
Let $\mf(G)$ be the category the objects of which 
are subgroups of $G$ with morphisms given by
\begin{equation}
\label{eq:mormf}
\Hom_{\mf(G)}(U,V)=\Hom_G(\Z[G/U],\Z[G/V]).
\end{equation}
Then $\mf(G)$ is a $\Z$-order category which is generated by the morphisms
\begin{align}
\rho^U_g\colon&\Z[G/{}^gU]\longrightarrow\Z[G/U],&
\rho_g^U(xgUg^{-1})&=xgU;\label{eq:mormf11}\\
\eui_{V,U}\colon&\Z[G/V]\longrightarrow\Z[G/U],&
\eui_{V,U}(xV)&=xU;\label{eq:mormf12}\\
\eut_{U,V}\colon&\Z[G/U]\longrightarrow\Z[G/V],&
\eut_{U,V}(xU)&=\textstyle{\sum_{r\in\caR}} xrV; \label{eq:mormf13}
\end{align}
$g\in G$, $U,V\subseteq G$, $V\subseteq U$, where $\caR\subseteq U$ is a set
of right $V$-coset representatives. 
The assignment 
\begin{equation}
\label{eq:sigmamac}
\sigma(U)=U,\ 
\sigma(\rho_g^U)=\rho_{g^{-1}}^{{}^gU} ,\
\sigma(\eui_{V,U})=\eut_{U,V},\ 
\sigma(\eut_{U,V})=\eui_{V,U},
\end{equation}
for $U,V\subseteq G$, $V\subseteq U$, $g\in G$, defines an
antipode $\sigma\colon\mf(G)\to\mf(G)^{op}$.
Let $\mf_R(G)$ denote the $R$-order category
obtained from $\mf(G)$ by tensoring with $R$.
Assigning to every cohomological $G$-Mackey functor $\boX$
with values in $\RMod$ the contravariant functor
$\tboX$ given by
\begin{equation}
\label{eq:maccon1}
\tboX(U)=\boX_U,\ 
\tboX(\rho^U_g)=c_{g,U}^\boX,\ 
\tboX(\eui_{V,U})=i_{U,V}^\boX,\ 
\tboX(\eut_{U,V})=t_{V,U}^\boX,
\end{equation}
yields an identification between $\MC_G(\RMod)$ and
$\euF_R(\mf_R(G)^{\op},\RMod)$.
Note that some authors prefer to identify the category
of cohomological Mackey functors with the category of covariant functors
of $\mf_R(G)$. The existence of the antipode $\sigma\colon\mf_G(R)\to\mf_G(R)^{\op}$
showes that both approaches are equivalent.

\subsection{The cohomological Mackey functors $\buT$ and $\boT$}
\label{ss:ST}
Let $G$ be a finite group.
There are two particular cohomological $G$-Mackey functors based on
the $R$-module $R$.
Let $\buT\in\ob(\MC_G(\RMod))$ be given by
\begin{equation}
\label{eq:defbuT1}
\buT_U=R,\ \ i_{U,V}^{\buT}=|U:V|\iid_R,\ \ t_{V,U}^{\buT}=\iid_R,\ \ 
c_{g,U}^{\buT}=\iid_R,
\end{equation}
and $\boT\in\ob(\MC_G(\RMod))$ be given by
\begin{equation}
\label{eq:defboT}
\boT_U=R,\ \ i_{U,V}^\boT=\iid_R,\ \ t_{V,U}^\boT=|U:V|\iid_R,\ \ 
c_{g,U}^\boT=\iid_R,
\end{equation}
for $U,V\subseteq G$, $V\subseteq U$.
Then $\buT$ and $\boT$ are $R$-lattice functors, and one has
 $\boT\simeq\buT^\ast$.
 
Let $R$ be an integral domain of characteristic $0$. 
For such a ring the subfunctor
$\bSigma\subseteq\boT$ given by 
$\bSigma_U= |U|\cdot\boT_U$
is canonically isomorphic to $\buT$, i.e.,
there exists a canonical injective natural transformation
$j\colon\buT\to\boT$. We denote by $\boB=\coker(j)$
the cokernel of this canonical map. 

Let $\boX\in\ob(\MC_G(\RMod))$, and let $\phi\colon\boT\to\boX$
be a natural transformation. Then $\phi$ is uniquely determined by
$\phi_G\colon\boT_G\to\boX_G$, and every such morphism
defines a unique natural transformation $\phi\colon\boT\to\boX$.
Hence one has a canonical isomorphism
\begin{equation}
\label{eq:homT}
\nat_G(\boT,\boX)\simeq\boX_G.
\end{equation}
In a similar fashion one shows that 
\begin{equation}
\label{eq:homuT}
\nat_G(\buT,\boX)\simeq\boX_{\{1\}}^G.
\end{equation}

\subsection{Invariants and coinvariants}
\label{ss:invcoinv}
There are two
standard procedures which turn a left $\RG$-module $M$
into a cohomological $G$-Mackey functor with values in the category
of $R$-modules. By $\boh^0(M)$ we denote 
what is called the {\it fixed-point-functor} in \cite{pj:simp}.
In more detail, one has
$\boh^0(M)_U=M^U$, for $U,V\subseteq G$, $V\subseteq U$,
$i_{U,V}^{\boh^0(M)}\colon M^U\to M^V$ is the canonical map,
$t_{V,U}^{\boh^0(M)}\colon M^V\to M^U$ is given by the transfer,
i.e., if $\caR\subseteq U$ denotes a system of coset representative of $U/V$
then $t_{V,U}^{\boh^0(M)}$ is given by multiplication with $\sum_{r\in\caR}r$,
and $c_{g,U}^{\boh^0(M)}\colon M^U\to M^{{}^gU}$ is left-multiplication by $g\in G$.

By $\boh_0(M)$ we denote the cohomological $G$-Mackey functors of {\it coinvariants}.
Thus $\boh_0(M)_U=M/\omega_{R[U]}M$, where $\omega_{R[U]}=\kernel(R[U]\to R)$
is the augmentation ideal in $R[U]$,
and for $U,V\subseteq G$, $V\subseteq U$, 
$t_{V,U}^{\boh_0(M)}\colon M_V\to M_U$ is the canonical map,
the map $i_{U,V}^{\boh_0(M)}\colon M_U\to M_V$ is induced by multiplication with
$\sum_{r\in\caR} r^{-1}$, and the map $c_{g,U}^{\boh^0(M)}\colon M_U\to M_{{}^gU}$ is
induced by multiplication with $g\in G$. E.g., one has canonical isomorphisms
of cohomological $G$-Mackey functors
$\buT\simeq\boh_0(R)$ and $\boT\simeq\boh^0(R)$, where
$R$ denotes the trivial left $\RG$-module.

\subsection{Standard projective cohomological Mackey functors}
\label{ss:projmac}
By \S\ref{ss:proj}, one knows that for $W\subseteq G$ the functor
\begin{equation}
\label{eq:standproj}
\boP^W=\Hom_G(R[G/\argu],R[G/W])\in\ob(\MC_G(\RMod)),\ \ 
\end{equation}
where $\boP^W_U=\Hom_G(R[G/U],R[G/W])=R[G/W]^U$,
is projective in $\MC_G(\RMod)$.
These functors can be described as follows.

\begin{fact}
\label{fact:projcohmac}
Let $G$ be a finite group, and let $W\subseteq G$.
Then one has canonical isomorphisms
\begin{equation}
\label{eq:standproj2}
\boP^W\simeq\boh^0(R[G/W])\simeq\boh^0(\boP^W_{\{1\}})\simeq\Midn_W^G(\boT^W),
\end{equation}
where $\Midn_W^G(\argu)$ denotes the induction functor
in the category of Mackey functors \textup{(cf. \cite[\S 4]{pj:simp})},
and $\boT^W\in\ob(\MC_W(\RMod))$ is the cohomological
$W$-Mackey functor described in subsection~\ref{ss:ST}.
\end{fact}

Both descriptions of the standard projective cohomological $G$-Mackey functors
will be useful for our purpose.
Note that one has canonical isomorphisms $\boP^G\simeq\boT$,
i.e., $\boT$ is projective. 
We also put $\boQ=\boP^{\{1\}}=\boh^0(\RG)$.

\begin{rem}
\label{rem:minproj}
Let $G$ be a finite $p$-group, and let $R$ be a discrete valuation
domain of characteristic $0$ with maximal ideal $pR$.
For $W\subseteq G$ there exists a simple cohomological $G$-Mackey functor
$\boS^W$ with values in the category of $R$-modules given by
\begin{equation}
\label{eq:simpmacfun}
\boS^W_U=
\begin{cases}
\F&\ \text{for $U= {}^gW$;}\\
0&\ \text{for $U\not= {}^gW$.}
\end{cases}
\end{equation}
In particular, for $U,V\subseteq G$, $V\subsetneq U$, one has
$i_{U,V}^{\boS^W}=0$ and $t_{V,U}^{\boS^W}=0$.
Moreover, any simple cohomological $G$-Mackey functor is 
isomorphic to some $\boS^W$, $W\subseteq G$ (cf. \cite{pj:simp}).
The Nakayama relations and \eqref{eq:homT}
show that for $V\subseteq G$ and $V\not={}^gW$ one has
\begin{equation}
\label{eq:simp1}
\nat_G(\boP^V,\boS^W)=\nat_V(\boT,\rst^G_V(\boS^W))\simeq \boS^W_V=0.
\end{equation}
On the other hand for $V={}^gW$ one has
\begin{equation}
\label{eq:simp2}
\nat_G(\boP^V,\boS^W)=\nat_V(\boT,\rst^G_V(\boS^W))\simeq \boS^W_V=\F.
\end{equation}
Hence $\boP^W$ is the (minimal) projective cover of $\boS^W$
for all $W\subseteq G$.
\end{rem}

\subsection{Standard relative injective cohomological Mackey functors}
\label{ss:standinj}
Let $G$ be a finite group, and $W\subseteq G$.
The functor $\Midn_W^G(\argu)$ commutes with the functor $\argu^\ast$
on $R$-lattice functors, i.e., one has a natural isomorphism
\begin{equation}
\label{eq:natisoast}
\Midn_W^G(\argu^\ast)\simeq\Midn_W^G(\argu)^\ast\colon
\MC_W(\Rlat)\longrightarrow\MC_G(\Rlat)^{\op}.
\end{equation}
Thus 
$\boJ^W=(\boP^W)^\ast\simeq\Midn_W^G(\buT)\in\ob(\MC_G(\Rlat))$.

\subsection{The Yoneda dual}
\label{ss:yonmac}
Let $\boX\in\ob(\MC_G(\Rlat))$. 
As $\boP^W$, $W\subseteq G$, takes values in the category of $R$-lattices,
the Nakayama relations and \eqref{eq:homuT}
yield canonical isomorphisms
\begin{equation}
\label{eq:yondumac}
\begin{aligned}
\nat_G(\boX,\boP^W)&\simeq\nat_G(\boJ^W,\boX^\ast)
\simeq\nat_G(\Midn_W^G(\buT),\boX^\ast)\\
&\simeq\nat_W(\buT,\rst^G_W(\boX^\ast))
\simeq(\boX_{\{1\}}^\ast)^W
\end{aligned}
\end{equation}
From this one concludes the following property (cf. Fact~\ref{fact:projcohmac}).

\begin{fact}
\label{fact:yondumac}
Let $\boX\in\ob(\MC_G(\Rlat))$. Then
$\boX^\circledast\simeq\boh^0(\boX^\ast_{\{1\}})$. 
\end{fact}

\section{Section cohomology of cohomological Mackey functors}
\label{s:secmac}
If not stated otherwise 
$R$ will denote a commutative ring with unit $1_R\in R$.
Let $G$ be a finite group, and let $\boX\in\ob(\MC_G(\RMod))$.
For $U,V\subseteq G$, $V\triangleleft U$, one defines the
{\it section cohomology groups} of $\boX$ by
\begin{equation}
\label{eq:seccoh}
\begin{aligned}
\bok^0(U/V,\boX)&=\kernel(i_{U,V}^\boX),&\bok^1(U/V,\boX)&=\boX_V^U,\\
\boc_0(U/V,\boX)&=\coker(t_{V,U}^\boX),&\boc_1(U/V,\boX)&=\kernel(t_{V,U}^\boX)/\omega_{U/V}\boX_V.
\end{aligned}
\end{equation}
The following properties were established in \cite[\S 2.4]{tw:stand}.

\begin{prop}
\label{prop:sec}
Let $G$ be a finite group, let $U$ and $V$ be subgroups of $G$ such that
$V$ is normal in $U$, and let $\boX$ be a cohomological $G$-Mackey functor with values
in $\RMod$.
\begin{itemize}
\item[(a)] The canonical maps yield an exact sequence of $R$-modules
\begin{equation}
\label{eq:6term}
\xymatrix{
0\ar[r]&\boc_1(U/V,\boX)\ar[r]&
\hH^{-1}(U/V,\boX_V)\ar[r]&
\bok^0(U/V,\boX)\ar[d]\\
0&\bok^1(U/V,\boX)\ar[l]&\hH^{0}(U/V,\boX_V)\ar[l]&\boc_0(U/V,\boX)\ar[l]\\
}
\end{equation}
where $\hH^\bullet(U/V,\argu)$ denotes the Tate cohomology groups.
\item[(b)] Let $\xymatrix{0\ar[r]&\boX\ar[r]^\phi&\boY\ar[r]^\psi&\boZ\ar[r]&0}$ be a short exact sequence
of cohomological $G$-Mackey functors with values in $\RMod$. Then one has exact sequences
\begin{equation}
\label{eq:longk}
\xymatrix@R=3pt{
0\ar[r]&\bok^0(U/V,\boX)\ar[r]^{\bok^0(\phi)}&\bok^0(U/V,\boY)\ar[r]^{\bok^0(\psi)}&
\bok^0(U/V,\boZ)\ar[r]&\ldots\\
\ldots\ar[r]&\bok^1(U/V,\boX)\ar[r]^{\bok^1(\phi)}&\bok^1(U/V,\boY)\ar[r]^{\bok^1(\psi)}&
\bok^1(U/V,\boZ)
}
\end{equation}
and
\begin{equation}
\label{eq:longc}
\xymatrix@R=3pt{
&\boc_1(U/V,\boX)\ar[r]^{\boc_1(\phi)}&\boc_1(U/V,\boY)\ar[r]^{\boc_1(\psi)}&
\boc_1(U/V,\boZ)\ar[r]&\ldots\\
\ldots\ar[r]&\boc_0(U/V,\boX)\ar[r]^{\boc_0(\phi)}&\boc_0(U/V,\boY)\ar[r]^{\boc_0(\psi)}&
\boc_0(U/V,\boZ)\ar[r]&0.
}
\end{equation}
\end{itemize}
\end{prop}

\subsection{Section cohomology for cyclic subgroups}
\label{ss:cycsec}
Let $W$ be a non-trivial cyclic subgroup of the finite group $G$ generated by the element
$w\in W$. 
Taking coinvariants of the chain complex of $R[W]$-modules
$R[W]\overset{w-1}{\longrightarrow} R[W]$ yields an exact sequence
\begin{equation}
\label{eq:seccyc2}
\xymatrix@C=1.3truecm{
0\ar[r]&\boT^W\ar[r]^{\boP(\eut_{W,\{1\}})}&\boQ^W\ar[r]^{\boh^0(w-1)}&\boQ^W\ar[r]&\buT^W\ar[r]&0},
\end{equation}
(cf. \S\ref{ss:ST}, \S\ref{ss:projmac}).
If $R$ is an integral domain of characteristic $0$, one has
additionally a short exact sequence
\begin{equation}
\label{eq:seccyc1}
\xymatrix{
0\ar[r]&\buT^W\ar[r]&\boT^W\ar[r]&\boB^W\ar[r]&0}.
\end{equation}
Splicing together the short exact sequences \eqref{eq:seccyc2}
and \eqref{eq:seccyc1}
yields a projective resolution of the cohomological $W$-Mackey functor $\boB^W$.
Using this projective resolution,
Fact~\ref{fact:Pfirst}, \eqref{eq:homT} and \eqref{eq:homuT}
one concludes the following.

\begin{fact}
\label{fact:seccyc}
Let $R$ be a integral domain of characteristic $0$,
and let $W\subseteq G$ be cyclic subgroup of the finite group $G$. Then
for $k\in\{0,1\}$ one has canonical isomorphisms
\begin{equation}
\label{eq:seccyc3}
\begin{aligned}
\ext_G^k(\Midn_W^G(\boB^W),\boX)&\simeq \bok^k(W/\{1\},\boX),\\
\ext_G^{3-k}(\Midn_W^G(\boB^W),\boX)&\simeq \boc_k(W/\{1\},\boX).
\end{aligned}
\end{equation}
\end{fact}

Note that Fact~\ref{fact:seccyc} shows also that for a cyclic subgroup $W\subseteq G$ one 
can consider the groups $\bok^\bullet(W/\{1\},\argu)$, $\boc_{3-\bullet}(W/\{1\},\argu)$
together with the respective connecting homomorphisms as a cohomological functor 
(cf. \cite[\S XII.8]{mcl:hom}).

\subsection{Cohomological Mackey functors of type $H^0$ and $H_0$}
\label{ss:h0}
Let $G$ be a finite group, and let $\boX\in\ob(\MC_G(\RMod))$.
Then $\boX$ will be called {\it $i$-injective}, if for all $U,V\subseteq G$, $V\subseteq U$,
the map $i^\boX_{U,V}$ is injective; i.e., 
$\boX$ is $i$-injective if, and only if, for all 
$U,V\subseteq G$, $V\triangleleft U$, one has $\bok^0(U/V,\boX)=0$.
Moreover, $\boX$ will be called {\it of type $H^0$} (or to satisfy {\it Galois descent}), 
if $\boX$ is $i$-injective
and $\bok^1(U/V,\boX)=0$ for all $U,V\subseteq G$, $V\triangleleft U$, i.e., 
$\boX$ is of type $H^0$ if, and only if,
one has a canonical isomorphism (induced by $i$)
\begin{equation}
\label{eq:hHo0}
\boX\simeq\boh^0(\boX_{\{1\}}).
\end{equation}
The cohomological $G$-Mackey functor $\boX$ will be called to be
{\it Hilbert$^{90}$}, if it is of type $H^0$ and
$H^1(U,\boX_{\{1\}})=0$ for every subgroup $U$ of $G$.
One has the following property.

\begin{prop}
\label{prop:H90}
Let $G$ be a finite group, and let $\boX\in\ob(\MC_G(\RMod))$ be 
Hilbert$^{90}$. Then for all $U,V\subset G$, $V\triangleleft U$,
one has $H^1(U/V,\boX_V)=0$.
\end{prop}

\begin{proof}
By the 5-term exact sequence, inflation $H^1(U/V,\boX_{\{1\}}^V)\to H^1(U,\boX_{\{1\}})$
is injective. 
Hence $H^1(U/V,\boX_{\{1\}}^V)=0$.
As $\boX$ is of type $H^0$, $\boX_{\{1\}}^V$  and $\boX_V$ are isomorphic $R[U/V]$-modules.
This yields the claim.
\end{proof}

In a similar fashion one calls $\boX$ to be {\it $t$-surjective},
if for all $U,V\subseteq G$, $V\subseteq U$,
the map $t^\boX_{V,U}$ is surjective; i.e., 
$\boX$ is $t$-surjective if, and only if, for all 
$U,V\subseteq G$, $V\triangleleft U$, one has $\boc_0(U/V,\boX)=0$.
The cohomological $G$-Mackey functor 
$\boX$ will be called {\it of type $H_0$} (or to satisfy {\it Galois co-descent}), if $\boX$ is $t$-surjective
and $\boc_1(U/V,\boX)=0$ for all $U,V\subseteq G$, $V\triangleleft U$, i.e., 
$\boX$ is of type $H_0$ if, and only if,
one has a canonical isomorphism (induced by $t$)
\begin{equation}
\label{eq:hHu0}
\boh_0(\boX_{\{1\}})\simeq\boX.
\end{equation}
Furthermore, $\boX$ will be called to be
{\it co-Hilbert$^{90}$}, if it is of type $H_0$ and for every subgroup $U$ of $G$ one has
$\hH^{-1}(U,\boX_{\{1\}})=0$.

\begin{rem}
\label{rem:projH0}
Every projective cohomological $G$-Mackey functor $\boP$
with values in the category of $R$-modules
is a direct summand of a coproduct of standard projective
cohomological $G$-Mackey functors.
Hence by Fact~\ref{fact:projcohmac} every 
projective cohomological $G$-Mackey functor $\boP$
is of type $H^0$.
However, if $R$ is an integral domain of characteristic $0$,
the Nakayama relations imply that
$H^1(G,R[\Omega])=0$ for any $G$-set $\Omega$.
In particular, $\boP$ is even Hilbert$^{90}$.
\end{rem}

The periodicity of Tate cohomology for finite cyclic groups has the following
consequence.

\begin{prop}
\label{prop:H90c}
Let $G$ be a finite group, and let $\boX\in\ob(\MC_G(\RMod))$ be
Hilbert$^{90}$. Let $U,V\subseteq G$, $V\triangleleft U$, be such that
$U/V$ is cyclic. Then $\boc_1(U/V,\boX)=0$.
\end{prop}

\begin{proof}
By Proposition~\ref{prop:H90} and the periodicity of Tate cohomology (of period $2$), one has
$\hH^{-1}(U/V,\boX_V)\simeq H^1(U/V,\boX_V)=0$.
Hence \eqref{eq:6term} yields the claim.
\end{proof}

\subsection{Tate duality}
\label{ss:tatedual}
Let $R$ be a principal ideal domain of characteristic $0$,
and let $K=\quot(R)$ denote its quotient field.
Then $\II=K/R$ is an injective $R$-module\footnote{This follows by an argument
similar to the proof of \cite[Cor.~III.7.3]{mcl:hom}.}
Let $G$ be a finite group, and let $M$ be a left $\RG$-lattice. Then one has
an exact sequence of left $\RG$-modules
\begin{equation}
\label{eq:tatedual1}
\xymatrix{
0\ar[r]&M^\ast\ar[r]&\Hom_R(M,K)\ar[r]&\Hom_R(M,\II)\ar[r]&0,
}.
\end{equation}
where $M^\ast=\Hom_R(M,R)$.
The following property is also known as {\it Tate duality}.

\begin{prop}
\label{prop:tatedual}
Let $R$ be a principal ideal domain of characteristic $0$,
let $K=\quot(R)$ be the quotient field of $R$, and let $\II=K/R$.
Let $G$ be a finite group, and let $M$ be an $\RG$-lattice.
Then for all $k\in\Z$ one has natural isomorphisms
\begin{equation}
\label{eq:tatedual2}
\hH^k(G,M^\ast)\simeq\Hom_R(\hH^{-k}(G,M),\II).
\end{equation}
\end{prop}

\begin{proof}
It is well known that one has natural isomorphisms
\begin{equation}
\label{eq:tatedual3}
\hH^{k-1}(G,\Hom_R(M,\II))\simeq\Hom_R(\hH^{-k}(G,M),\II)
\end{equation}
for all $k\in\Z$ (cf. \cite[p.~148, Ex.~VI.7.4]{brown:coh}).
Moreover, as $\hH^k(G,\Hom_R(M,K))=0$ for all $k\in\Z$, one has also natural
isomorphisms
\begin{equation}
\label{eq:tatedual4}
\hH^{k-1}(G,\Hom_R(M,\II))\simeq \hH^k(G,M^\ast).
\end{equation}
This yields the claim.
\end{proof}


\subsection{Section cohomology of $R$-lattice functors}
\label{ss:seclat}
Let $R$ be an integral domain of characteristic $0$,
let $G$ be a finite group,
and let $\boX\in\ob(\MC_G(\Rlat))$ be an $R$-lattice functor.
For $U,V\subseteq G$, $V\subseteq U$, the axiom (cMF$_7$) (cf. \S\ref{ss:Gmac1})
implies that $\boX$ is $i$-injective. Hence by \eqref{eq:6term} one has
an isomorphism 
\begin{equation}
\label{eq:isolat}
\boc_1(U/V,\boX)\simeq\hH^{-1}(U/V,\boX_V)
\end{equation}
and a short exact sequence
\begin{equation}
\label{eq:23term}
\xymatrix{
0\ar[r]&\boc_0(U/V,\boX)\ar[r]&\hH^0(U/V,\boX_V)\ar[r]&\bok^1(U/V,\boX)\ar[r]&0.
}
\end{equation}
Let $R$ be a principal ideal domain,
and let $\phi\colon A\to B$ be a homomorphism of $R$-lattices.
Then $\phi$ is split injective if, and only if, $\phi^\ast\colon B^\ast\to A^\ast$ is surjective.
From this fact one concludes the following properties.

\begin{prop}
\label{prop:latsec}
Let $R$ be a principal ideal domain of characteristic $0$, let
$G$ be a finite group, and let $\boX\in\ob(\MC_G(\Rlat))$.
\begin{itemize}
\item[{\rm (a)}]  $\boX$ is of type $H^0$ if, and only if, $\boX^\ast$ is $t$-surjective.
\item[{\rm (b)}] The following are equivalent:
\begin{itemize}
\item[{\rm (i)}] $\boX$ is Hilbert$^{90}$;
\item[{\rm (ii)}] $\boX^\ast$ is of type $H_0$;
\item[{\rm (iii)}] $\boX^\ast$ is co-Hilbert$^{90}$.
\end{itemize}
\end{itemize}
\end{prop}

\begin{proof} 
Let $U,V\subseteq G$, $V\subseteq U$.

\noindent
(a) The map $i_{U,V}\colon \boX_U\to\boX_V$ is split-injective if, and only if,
$\bok^1(U/V,\boX)=0$. Hence the previously mentioned remark yields the claim.

\noindent
(b) Suppose that $\boX$ is Hilbert$^{90}$. Then
$H^1(U/V,\boX_V)=0$ for all $U,V\subseteq G$, $V\triangleleft U$ (cf. Prop.~\ref{prop:H90}). 
By Tate duality (cf. Prop.~\ref{prop:tatedual}), one has 
\begin{equation}
\label{eq:tatedual5}
\hH^{-1}(U/V,\boX_V^\ast)\simeq\Hom_R(H^1(U/V,\boX_V),\II_R)=0
\end{equation}
whenever $V$ is normal in $U$. Hence $\boc_1(U/V,\boX^\ast)=0$
for all $U,V\subseteq G$, $V\triangleleft U$
(cf. \eqref{eq:isolat}). Thus by (a), $\boX^\ast$ is of type $H_0$.
If $\boX^\ast$ is of type $H_0$, 
\eqref{eq:isolat} implies that $\hH^{-1}(U/V,\boX_V^\ast)=0$
for all $U,V\subseteq G$, $V\triangleleft U$, i.e., $\boX^\ast$ is co-Hilbert$^{90}$.
If $\boX^\ast$ is co-Hilbert$^{90}$, then
(a) implies that $\boX$ is of type $H^0$.
By Tate duality (cf. Prop.~\ref{prop:tatedual}), one has 
\begin{equation}
\label{eq:tatedual6}
H^{1}(U/V,\boX_V)\simeq\Hom_R(\hH^{-1}(U/V,\boX_V^\ast),\II_R)=0
\end{equation}
This yields the claim.
\end{proof}

\subsection{Finite cyclic groups}
\label{ss:fincyc}
If $G$ is a finite group and $R$ is any commutative ring with unit $1$,
one has $\boP^{\{1\}}\simeq (\boP^{\{1\}})^\ast$,
i.e., $\boP^{\{1\}}$ is projective and relative injective.
If $G$ is a finite
cyclic group, and $W\subseteq G$ is a non-trivial subgroup
of $G$, applying $\Midn_W^G(\argu)$ 
to the exact sequence \eqref{eq:seccyc2}
yields an exact sequence 
\begin{equation}
\label{eq:exact2}
\xymatrix@C=1.6truecm{
0\ar[r]&\boP^W\ar[r]^{\boP(\eut_{W,\{1\}})}
&\boP^{\{1\}}\ar[r]^{\Midn^G_W(w-1)}&\boP^{\{1\}}\ar[r]&\boJ^W\ar[r]&0,
}
\end{equation}
where $w\in W$ is a generating element of $W$.
In particular, 
\begin{equation}
\label{eq:projinjdim}
\prdim(\boJ^W)\leq 2,\qquad W\subseteq G,\ W\not=\{1\},
\end{equation}
and $\prdim(\boJ^{\{1\}})=0$. Thus one has (cf. \S\ref{ss:gor}).

\begin{prop}
\label{prop:gorcm}
Let $R$ be a Dedekind domain, and let $G$ be a finite cyclic group.
Then $\mf_R(G)$ is $2$-Gorenstein.
\end{prop}

For  $\ext_G^k(\boJ^W,\boX)=\caR^k\!\nat_G(\boJ^W,\boX)$,
$\boX\in\ob(\MC_G(\RMod))$, $W\subseteq G$,  one obtains the following.

\begin{prop}
\label{prop:extJ}
Let $R$ be a Dedekind domain,
let $G$ be a finite cyclic group, and let $\boX\in\ob(\MC_G(\RMod))$. 
Then for $W\subseteq G$ one has
\begin{itemize}
\item[(i)] $\ext^0_G(\boJ^W,\boX)=\nat_G(\boJ^W,\boX)\simeq\boX_{\{1\}}^W$;
\item[(ii)] $\ext^1_G(\boJ^W,\boX)\simeq\boc_1(W/\{1\},\boX)$;
\item[(iii)] $\ext^2_G(\boJ^W,\boX)\simeq \boc_0(W/\{1\},\boX)$;
\end{itemize}
and $\ext^k_G(\boJ^W,\boX)=0$ for $k\geq 3$.
\end{prop}

\begin{proof}
For $W=\{1\}$, one has $\ext^1_G(\boJ^W,\boX)=\ext^2_G(\boJ^W,\boX)=0$, and
$\ext^0_G(\boJ^W,\boX)=\nat_G(\boJ^W,\boX)=\boX_{\{1\}}$ (cf. Fact~\ref{fact:Pfirst}).
Hence the claim holds in this case, and we may assume that $W\not=\{1\}$.
From the Nakayama relations and \eqref{eq:exact2} one concludes that
$\ext^k_G(\boJ^W,\boX)$ coincides with the $k^{th}$-cohomology of the cochain complex
\begin{equation}
\label{eq:exact3}
\xymatrix{
&0\ar[r]&\boX_{\{1\}}\ar[r]^{w-1}&\boX_{\{1\}}\ar[r]^{t^\boX_{\{1\},W}}&\boX_W\ar[r]&0}
\end{equation}
concentrated in degrees $0$, $1$ and $2$.
This yields the claim in case that $W\not=\{1\}$.
\end{proof}

From Proposition~\ref{prop:extJ} one obtains the following description of the higher
derived functors of the Yoneda dual.

\begin{prop}
\label{prop:Ryon}
Let $R$ be a principal ideal domain of characteristic $0$, 
let $G$ be a finite cyclic group, and let $\boX\in\ob(\MC_G(\Rlat))$
be a cohomological $G$-Mackey functor with values in the category of $R$-lattices.
Then the following are equivalent.
\begin{itemize}
\item[(i)] $\boX$ is Hilbert$^{90}$;
\item[(ii)] $\boX^\ast$ is co-Hilbert$^{90}$;
\item[(iii)] $\boX$ is $\circledast$-acyclic;
\item[(iv)] $\boX^\circledast$ is Hilbert$^{90}$;
\item[(v)] $\boX^\circledast$ is $\circledast$-acyclic.
\end{itemize}
In particular, $(\mf_R(G),\sigma)$ is a $\circledast$-symmetric $R^\circledast$-order category.
\end{prop}

\begin{proof}
By Proposition~\ref{prop:latsec}(b), (i) and (ii) are equivalent.
For $W\subseteq G$ one has
\begin{equation}
\label{eq:acyc1}
\caR^k(\boX)^\circledast_W=\ext_G^k(\boX,\boP^W)\simeq\ext_G^k(\boJ^W,\boX^\ast).
\end{equation}
Hence Proposition~\ref{prop:extJ} implies that (ii) and (iii) are equivalent,
and thus also (iv) and (v) are equivalent.
By Fact~\ref{fact:yondumac}, $\boX^\circledast\simeq\boh^0(\boX^\ast_{\{1\}})$.
Let $W\subseteq G$. The periodicity
of Tate cohomology (or period 2) and Tate duality (cf. \eqref{eq:tatedual2}) imply that
\begin{equation}
\label{eq:tatedual7}
H^1(W,\boX_{\{1\}}^\ast)\simeq
\hH^{-1}(W,\boX_{\{1\}}^\ast)\simeq\Hom_R(H^1(W,\boX_{\{1\}}),\II_R).
\end{equation}
Hence (i) implies (iv). Replacing $\boX$ by $\boX^\circledast$ shows that (iv) implies (i). 
This yields the claim.
\end{proof}

The following property will allow us to analyze
the projective dimensions of cohomological
Mackey functors for finite cyclic groups.

\begin{prop}
\label{prop:cycint}
Let $R$ be a Dedekind domain of characteristic $0$, let
$G$ be a finite cyclic group, and let $\phi\colon\boP\to\boX$
be a surjective natural transformation in $\MC_G(\RMod)$,
where $\boP$ is a projective $R$-lattice functor. Then
\begin{itemize}
\item[(a)] $\kernel(\phi)$ is an $R$-lattice functor;
\item[(b)] if $\boX$ is $i$-injective, $\kernel(\phi)$ is of type $H^0$;
\item[(c)] if $\boX$ is of type $H^0$, $\kernel(\phi)$ is Hilbert$^{90}$.
\end{itemize}
\end{prop}

\begin{proof}
For (a) there is nothing to prove.
Put $\boK=\kernel(\phi)$, and let $U,V\subseteq G$, $V\subseteq U$. 
By Remark~\ref{rem:projH0} and Fact~\ref{fact:seccyc}, one has an exact sequence
\begin{equation}
\label{eq:seccyc4}
\xymatrix{
\boc_1(U/V,\boX)\ar[d]&0\ar[l]&\\
\boc_0(U/V,\boK)\ar[r]&
\boc_0(U/V,\boP)\ar[r]&
\boc_0(U/V,\boX)\ar[r]&0
}
\end{equation}
and isomorphisms
\begin{align}
\bok^0(U/V,\boX)&\simeq\bok^1(U/V,\boK),\label{eq:seccyc5}\\
\bok^1(U/V,\boX)&\simeq\boc_1(U/V,\boK).\label{eq:seccyc6}
\end{align}
Hence \eqref{eq:seccyc5} implies (b). If $\boX$ is of type $H^0$,
\eqref{eq:seccyc6} yields that $\boc_1(U/V,\boK)=0$.
Thus by \eqref{eq:isolat} and the periodicity of Tate cohomology (of period $2$),
one has
\begin{equation}
\label{eq:seccyc7}
H^1(U/V,\boK_V)\simeq \hH^{-1}(U/V,\boK_V)\simeq\boc_1(U/V,\boK)=0.
\end{equation}
This yields the claim.
\end{proof}

The following property will turn out to be useful for our purpose.

\begin{prop}
\label{prop:torhil}
Let $R$ be an integral domain of characteristic $0$,
and let $p\in R$. Assume further that $G$ is a finite cyclic group,
and that $\boX\in\ob(\MC_G(\Rlat))$ is an $R$-lattice functor
with the Hilbert$^{90}$ property. Then
$\boY=\boX/p\boX$ is of type $H^0$.
\end{prop}

\begin{proof}
We may suppose that $p\not=0$. Then $p\boX\simeq\boX$.
Let $U,V\in G$, $V\subseteq U$. By Fact~\ref{fact:seccyc}, one has
a long exact sequence
\begin{equation}
\label{eq:seccyc8}
\xymatrix{
\bok^0(U/V,\boY)\ar[d]&\bok^0(U/V,\boX)\ar[l]&\bok^0(U/V,\boX)\ar[l]_{p}&0\ar[l]\\
\bok^1(U/V,\boX)\ar[r]^p&
\bok^1(U/V,\boX)\ar[r]&
\bok^1(U/V,\boY)\ar[d]&\\
\boc_1(U/V,\boY)\ar[d]&\boc_1(U/V,\boX)\ar[l]&\boc_1(U/V,\boX)\ar[l]_{p}\\
\boc_0(U/V,\boX)\ar[r]^p&
\boc_0(U/V,\boX)\ar[r]&
\boc_0(U/V,\boY)\ar[r]&0
}
\end{equation}
As $\bok^0(U/V,\boX)=\bok^1(U/V,\boX)=\boc_1(U/V,\boX)=0$,
one concludes that $\bok^0(U/V,\boY)=\bok^1(U/V,\boY)=0$. This yields the claim.
\end{proof}


\subsection{Injectivity criteria}
\label{ss:injcrit}
For a finite $p$-group $G$ there are useful criteria ensuring that
a homomorphism $\phi\colon\boX\to\boY$ of cohomological $G$-Mackey functors
is injective. These criteria are based on the following fact.

\begin{fact}
\label{fact:soctriv}
Let $G$ be a finite $p$-group, let $\F$ be a field of characteristic
$p$, and let $M$ be a non-trivial, finitely generated
left $\F[G]$-module. Let $B\subseteq M$ be an $\F[G]$-submodule
satisfying $B\cap M^G=0$. Then $B=0$.
\end{fact}

\begin{proof} The $\F$-algebra $\F[G]$ is artinian. 
Moreover, as every irreducible left $\F[G]$-module is isomorphic to the trivial left $\F[G]$-module,
for any finitely generated left $\F[G]$-module $B$ one has $\soc(B)=B^G$.
Hence the hypothesis implies $\soc_G(B)=0$. Thus $B=0$.
\end{proof}

From Fact~\ref{fact:soctriv} one concludes the following injectivity criterion.

\begin{lem}
\label{lem:injmod}
Let $G$ be a finite $p$-group, and  let $\F$ be a field of characteristic $p$. 
Suppose that for 
$\phi\colon\boX\to\boY\in\mor(\MC_G(\FMod))$ one has that
\begin{itemize}
\item[(i)] $\phi_G\colon \boX_G\to\boY_G$ is injective, and 
\item[(ii)] $\boX$ is of type $H^0$, and $\boY$ is $i$-injective.
\end{itemize}
Then $\phi$ is injective.
\end{lem}

\begin{proof}
By hypothesis (i), $i^{\boY}_{G,U}\circ\phi_G\colon\boX_G\to\boY_U$ is injective for
all $U\subseteq G$.
If $V\subseteq G$ is normal in $G$, one has a commutative diagram
\begin{equation}
\label{dia:inj1}
\xymatrix{
\boX_G\ar[r]^{\phi_G}\ar[d]_{i^{\boX}_{G,V}}&\boY_G\ar[d]^{i^{\boY}_{G,V}}\\
\boX_V\ar[r]^{\phi_V}&\boY_V
}
\end{equation}
As $\boX$ is of type $H^0$, $\image(i_{G,V}^{\boX})=\boX_V^{G/V}=\soc_{G/V}(\boX_V)$
(cf. \eqref{eq:seccoh}).
Moreover, since $\phi_V\circ i^{\boX}_{G,V}=i^{\boY}_{G,V}\circ\phi_G$ is injective,
$\phi_V\vert_{\soc_{G/V}(\boX_V)}\colon\soc_{G/V}(\boX_V)\to\boY_V$ is injective, i.e., 
$\kernel(\phi_V)\cap\soc_{G/V}(\boX_V)=0$. 
Thus by Fact~\ref{fact:soctriv}, $\kernel(\phi_V)=0$ and $\phi_V$ is injective.

Let $U$ be any subgroup of $G$, and let $V\subseteq U$ be a subgroup of $U$ which is normal in $G$.
By the previously mentioned remark one has a commutative diagram
\begin{equation}
\label{dia:inj2}
\xymatrix{
\boX_U\ar[r]^{\phi_U}\ar[d]_{i^{\boX}_{U,V}}&\boY_U\ar[d]^{i^{\boY}_{U,V}}\\
\boX_V\ar[r]^{\phi_V}&\boY_V
}
\end{equation}
with $\phi_V$ is injective. By hypothesis (ii), $i_{U,V}^{\boX}$
is injective. Hence $\phi_U$ is injective, and
this yields the claim.
\end{proof}

Let $R$ be a discrete valuation domain of characteristic $0$
with prime ideal $pR$ for some prime number $p$, i.e.,
$\F=R/pR$ is a field of characteristic $p$.
For a finitely generated $R$-module $A$
let
$\gr_\bullet(A)$ denote the graded $\F[t]$-module
associated to the p-adic filtration $(p^k.A)_{k\geq 0}$.
Then every homogeneous component $\gr_k(A)$ is
a finite-dimensional $\F$-vector space. 
Moreover, $A$ is a free $R$-module if, and only if, $\gr_\bullet(A)$
is a free $\F[t]$-module.
Let $\phi\colon A\to B$ be a homomorphism of 
finitely generated $R$-modules.
Then $\phi$ induces a homomorphism of $\F[t]$-modules
$\gr_\bullet(\phi)\colon\gr_\bullet(A)\to\gr_\bullet(B)$. 
Moreover, one has the following.

\begin{fact}
\label{fact:disgr}
Let $R$ be a discrete valuation domain of characteristic $0$
with prime ideal $pR$ for some prime number $p$,
and let $\phi\colon A\to B$ be a homomorphism of $R$-lattices.
Then the following are equivalent:
\begin{itemize}
\item[(i)] $\phi$ is split-injective;
\item[(ii)] $\gr_\bullet(\alpha)\colon\gr_\bullet(A)\to\gr_\bullet(B)$ is injective;
\item[(iii)] $\gr_0(\alpha)\colon\gr_0(A)\to\gr_0(B)$ is injective.
\end{itemize}
\end{fact}

Lemma~\ref{lem:injmod} and Fact~\ref{fact:disgr} imply the following criterion for split-injectivity.

\begin{prop}
\label{prop:inj}
Let $G$ be a finite $p$ group, let $R$ be a discrete valuation domain of characteristic $0$
with prime ideal $pR$, and let
$\phi\colon \boX\to\boY\in \mor(\MC_G(\Rlat))$ be a natural transformation
of cohomological $R$-lattice functors with the following properties:
\begin{itemize}
\item[(i)] $\gr_0(\phi_G)\colon \gr_0(\boX_G)\to\gr_0(\boY_G)$
is injective;
\item[(ii)] $\gr_0(\boX)$ is of type $H^0$ and $\gr_0(\boY)$ is $i$-injective.
\end{itemize}
Then $\phi\colon\boX\to\boY$ is split-injective.
\end{prop}


\section{Gentle $R^\circledast$-order categories}
\label{ss:gent}
Throughout this section we fix a prime number $p$
and assume further that $R$ is a principal ideal domain
of characteristic $0$ such that $pR$ is a prime ideal, i.e.,
$\F=R/pR$ is a field, the {\it residue field} of $R$ at $pR$.
By $K=\quot(R)$ we denote the quotient field of $R$.

\subsection{Gentle $R^\circledast$-order categories}
\label{ss:gentordcat}
By $\Gent_R(n,p)$, $n\geq 0$, we denote the $R$-order category with
objects $\ob(\Gent_R(n,p))=\{0,\ldots,n\}$ and morphisms given by
\begin{equation}
\label{eq:morgent}
\Hom_{\Gent_R(n,p)}(j,k)=
\begin{cases}
R.t_{j,k}&\ \text{for $j<k$,}\\
R.\iid_k&\ \text{for $j=k$,}\\
R.i_{j,k}&\ \text{for $j>k$,}
\end{cases}
\end{equation}
for $0\leq j,k\leq n$ subject to the relations
\begin{itemize}
\item[(i)] $i_{l,j}=i_{k,j}\circ i_{l,k} $ for $j\leq k\leq l$;
\item[(ii)] $t_{j,l}=t_{k,l}\circ t_{j,k} $ for $j\leq k\leq l$;
\item[(iii)] $i_{j+1,j}\circ t_{j,j+1}=p.\iid_j$ for $j\in \{0,\ldots,n-1\}$;
\item[(iv)] $ t_{k-1,k}\circ i_{k,k-1}=p.\iid_k$ for $k\in \{1,\ldots,n\}$;
\end{itemize}
where we put $t_{k,k}=i_{k,k}=\iid_k$ for $k\in\{0,\ldots,n\}$.
It comes equipped with the natural equivalence 
$\sigma\colon\Gent_R(n,p)\to\Gent_R(n,p)^{\op}$
of order $2$, i.e., $\sigma\circ\sigma=\iid_{\Gent_R(n,p)}$, given by
\begin{equation}
\label{eq:antiG}
\sigma(k)=k,\ \ \sigma(t_{j,k})=i_{k,j},\ \ \sigma(i_{k,j})=t_{j,k},\ \ 0\leq j\leq k\leq n;
\end{equation}
and thus forms an $R^\circledast$-order category. 

\begin{rem}
\label{rem:gentle}
Let $\mu=\mu_{\Gent_R(n,p)}$ be the $R$-order representing $\Gent_R(n,p)$
(cf. Remark~\ref{rem:rep}). Then $\mu\otimes_R\F$ is a gentle $\F$-algebra.
It is well known that these algebras are $1$-Gorenstein (cf. \cite{gere:gent}). However,
for $n\geq 1$ they are not of finite global dimension, and, therefore, they do not have the
Whitehead property (cf. Fact~\ref{fact:glwh} and \ref{fact:gldim1}).
\end{rem}

\subsection{The unitary projection}
\label{ss:unipro1}
Let $C_{p^n}$ be the cyclic group of order $p^n$. Then
\begin{equation}
\label{eq:unitar1}
\pi\colon\mf_R(C_{p^n})\longrightarrow\Gent_R(n,p),
\end{equation}
given by 
$\pi(U)=\log_p(|G:U|)$, $\pi(\eui_{V,U})=i_{j,i}$,
$\pi(\eut_{U,V})=t_{i,j}$, $\rho^U_g=\iid_i$,
for $U,V\subseteq G$, $|U|=p^{n-i}$, $|V|=p^{n-j}$, $j\geq i$,
is a unitary projection. Applying $\ifl^\pi(\argu)$ shows that every
functor $\boF\in\ob(\euF_R(\Gent_R(n,p)))$ can also be considered
as a cohomological Mackey functor for the finite group $C_{p^n}$.
The deflation functor $\dfl^\pi(\argu)$ can be described explicitly
using the functor of $C_{p^n}$-coinvariants $\argu_{C_{p^n}}$, i.e., 
for $\boX\in\ob(\MC_{C_{p^n}}(\RMod))$ one has
\begin{equation}
\label{eq:unitar2}
\dfl^\pi(\boX)(k)=(\boX_U)_{C_{p^n}},\ \ \text{$|U|=p^{n-k}$,}
\end{equation}
and $\dfl^\pi(\alpha)(k)=(\alpha_U)_{C_{p^n}}\colon(\boX_U)_{C_{p^n}}\to(\boY_U)_{C_{p^n}}$
for $\alpha\in\Hom_{\mf_R(C_{p^n})}(\boX,\boY)$.
Furthermore, by Fact~\ref{fact:dfl}(c), one has for $W\subseteq G$, $|W|=p^{n-k}$, that
\begin{equation}
\label{eq:dflps}
\dfl^\pi(\boP^W)\simeq \boP^k.
\end{equation}

\subsection{Simple functors}
\label{ss:simpgent}
As every functor $\boF\in\ob(\euF_R(\Gent_R(n,p)^{\op},\RMod))$
is in particular a cohomological $C_{p^n}$-Mackey functor,
one can use the description given in \cite{pj:simp}
in order to determine all simple functors in 
$\ob(\euF_R(\Gent_R(n,p)^{\op},\RMod^{\fg}))$. For every $\ell\in\{0,\ldots,n\}$
there exists a simple functor $\boS^\ell$ given by
\begin{equation}
\label{eq:simp}
\boS^\ell(k)=\begin{cases}
\F,&\ \text{for $k=\ell$,}\\
0,&\ \text{for $k\not=\ell$;}
\end{cases}
\ \ \boS^\ell(t_{j,k})=0,\ \ \boS^\ell(i_{k,j})=0,\ \ 
0\leq j\leq k\leq n.
\end{equation}
From Remark~\ref{rem:minproj}
one concludes that if $R$ is discrete valuation ring of
characteristic $0$ with maximal ideal $pR$, then
every simple functor $\boS\in\ob(\euF_R(\Gent_R(n,p)^{\op},\RMod^{\fg}))$
must be naturally isomorphic to some $\boS^\ell$, $0\leq \ell\leq n$.

\subsection{$R$-lattice functors of rank $1$}
\label{ss:rank1}
Let $\boF\in\ob(\euF_R(\Gent_R(n,p)^{\op},\Rlat))$ be an $R$-lattice functor.
Then $\boF(i_{0,k})\otimes_R K\colon \boF(k)\otimes_RK\to\boF(0)\otimes_RK$ is an 
isomorphism of finite-dimensional $K$-vector spaces,
i.e., $\rk(\boF(k))=\rk(\boF(0))$ for all $k\in\{0,\ldots,n\}$, where 
$\rk(\boF(0))$ denotes the rank of the free $R$-module $\boF(0)$.
We define the
{\it rank of $\boF$} by $\rk(\boF)=rk(\boF(0))$.

If $M$ is an $R$-lattice and $B\subseteq M$ is an $R$-submodule of $M$,
we denote by 
\begin{equation}
\label{eq:defsat}
\sat_M(B)=\{\,b\in M\mid\exists r\in R\setminus\{0\}\colon\ r.b\in B\,\}
\end{equation}
the {\it saturation} of $B$ in $M$. It is again an $R$-submodule of $M$.
Let $\boG$ be a subfunctor of the $R$-lattice functor $\boF\in\ob(\euF_R(\Gent_R(n,p)^{\op},\Rlat))$.
Then $\sat_{\boF}(\boG)$ given by
\begin{equation}
\label{eq:defsat2}
\sat_{\boF}(\boG)(k)=\sat_{\boF(k)}(\boG(k)),
\end{equation}
$0\leq k\leq n$, is a subfunctor of $\boF$ containing $\boG$.
The subfunctor $\boG$ will be called {\it saturated}, if $\sat_\boF(\boG)=\boG$.
The following fact allows us to reduce some considerations to
$R$-lattice functors of rank $1$.

\begin{fact}
\label{fact:rank11}
Let $\boF\in\ob(\euF_R(\Gent_R(n,p)^{\op},\Rlat))$, $\rk(\boF)> 0$. 
Then $\boF$ contains a
saturated subfunctor of rank $1$.
In particular, there exists an ascending chain $(\boF_j)_{0\leq j\leq \rk(\boF)}$
of subfunctors of $\boF$ satisfying
$\boF_0=0$, $\boF_{j-1}\subseteq\boF_j$, $\boF_{\rk(\boF)}=\boF$ and $\boF_j/\boF_{j-1}$ is an $R$-lattice functor of rank~$1$.
\end{fact}

\begin{proof}
Let $a\in\boF(0)$, $a\not=0$. Then $\boR\subseteq\boF$ 
given by $\boR(k)=\sat_{\boF(k)}(R\boF(i_{k,0})(a))$ together with the canonical maps 
is a saturated {subfunctor} of $\boF$. The final remark follows by induction.
\end{proof}

Let $\boF$ be an $R$-lattice functor of rank $1$.
By (iii) and (iv) of the definition, for $k\in\{0,\ldots n-1\}$ 
either $\boF(t_{k,k+1})$ is an isomorphism,
or $\boF(i_{k+1,k})$ is an isomorphism.
Thus we can represent $\boF$ by a diagram $\Delta_{\boF}$,
where we draw an arrow from
$k+1$ to $k$ if $\boF(t_{k,k+1})\colon\boF(k+1)\to\boF(k)$ is an isomorphism,
and an arrow from $k$ to $k+1$ if $\boF(i_{k+1,k})\colon\boF(k)\to\boF(k+1)$ is an
isomorphism. It is straightforward to verify that the isomorphism type of $\boF$
is uniquely determined by $\Delta_{\boF}$, and that for every arrow diagram
$\Delta$ there exists an $R$-lattice functor $\boF_\Delta$ which is represented
by this diagram.

\begin{rem}
\label{rem:diafun}
(a) For $\ell\in\{0,\ldots,n\}$ let $\boP^\ell=\Hom_{\Gent}(\argu,\ell)$ 
be the standard projective $R$-lattice functor associated
to $\ell$ (cf. \S\ref{ss:proj}). Then $\boP^\ell$ has rank $1$ and is represented by the arrow diagram
\begin{equation}
\label{eq:arrP}
\xymatrix@C.5truecm{
0&1\ar[l]&\cdots\ar[l]&\ell-1\ar[l]&\ell\ar[r]\ar[l]&\ell+1\ar[r]&\cdots\ar[r]&n-1\ar[r]&n
}.
\end{equation}

\noindent
(b) If $\boF$ is represented by the diagram $\Delta_{\boF}$, then $\boF^\ast$
is represented by the diagram $\Delta_{\boF^\ast}=\bar{\Delta}_{\boF}$, where all
arrows are reversed.

\noindent
(c) Let $\boJ^\ell=(\boP^\ell)^\ast$, $\ell\in\{0,\ldots,n\}$. Then $\boJ^\ell$ is relative
injective and, by (a) and (b), $\boJ^\ell$ is represented by the diagram
\begin{equation}
\label{eq:arrJ}
\xymatrix@C.5truecm{
0\ar[r]&1\ar[r]&\cdots\ar[r]&\ell-1\ar[r]&\ell&\ell+1\ar[l]&\cdots\ar[l]&n-1\ar[l]&n\ar[l]
}.
\end{equation}
In particular, $\boP^0\simeq\boJ^n$ and $\boP^n\simeq\boJ^0$ are relative injective.
\end{rem}

Let $\boF\in\ob(\euF_R(\Gent_R(n,p)^{\op},\Rlat))$ be a an $R$-lattice functor of rank $1$.
Then $\Delta_{\boF}$ defines a connected graph $\Gamma_{\boF}$ in the plane 
$\RR^2=\RR e_1\oplus\RR e_2$, where all
arrows are diagonal and point in negative $e_2$-direction, e.g., for $\boF\in
\ob(\euF_R(\Gent_R(8,p)^{\op},\Rlat))$ with $\Delta_\boF$ given by
\begin{equation}
\label{eq:arrex1}
\xymatrix@C.5truecm{
0&1\ar[l]\ar[r]&2\ar[r]&3&4\ar[r]\ar[l]&5\ar[r]&6\ar[r]&7&8\ar[l]
}
\end{equation}
one obtains the graph $\Gamma_\boF$
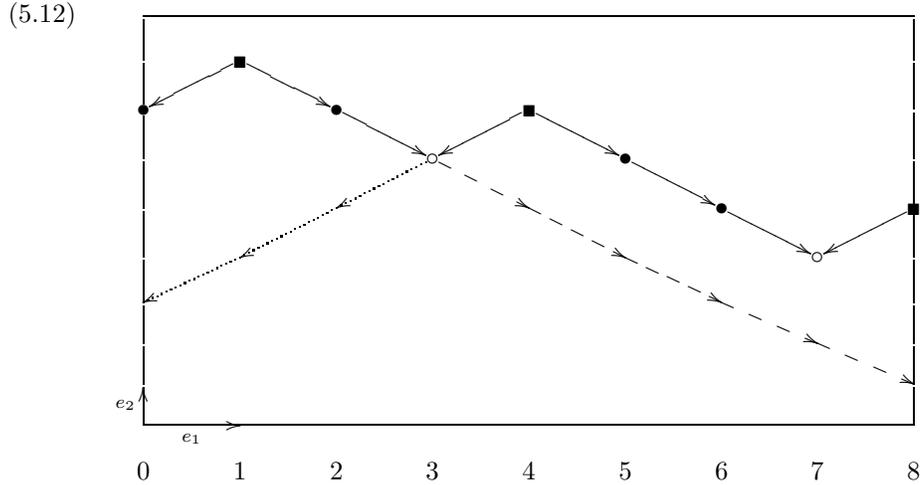
\begin{equation}
\label{eq:arrex2}
\xymatrix@C1.1truecm @M=0pt @W=0pt @R=.5cm{
\ar@{-}[r]\ar@{-}[d]&\ar@{-}[r]&\ar@{-}[r]&\ar@{-}[r]&\ar@{-}[r]&\ar@{-}[r]&\ar@{-}[r]&\ar@{-}[r]&\ar@{-}[d]\\
\ar@{-}[d]&\ar[dl]\ar[dr]\miniblacksquare&&&&&&&\ar@{-}[d]\\
\bullet\ar@{-}[d]&&\bullet\ar[dr]&&\miniblacksquare\ar[dl]\ar[dr]&&&&\ar@{-}[d]\\
\ar@{-}[d]&&&\circ\ar@{-->}[dr]\ar@{..>}[dl]&&\bullet\ar[dr]&&&\ar@{-}[d]\\
\ar@{-}[d]&&\ar@{..>}[dl]&&\ar@{-->}[dr]&&\bullet\ar[dr]&&\miniblacksquare\ar[dl]\ar@{-}[d]\\
\ar@{-}[d]&\ar@{..>}[dl]&&&&\ar@{-->}[dr]&&\circ&\ar@{-}[d]\\
\ar@{-}[d]&&&&&&\ar@{-->}[dr]&&\ar@{-}[d]\\
\ar@{-}[d]&&&&&&&\ar@{-->}[dr]&\ar@{-}[d]\\
&&&&&&&&\ar@{-}[d]\\
\ar@{->}[u]^{e_2}\ar@{->}[r]_{e_1}&\ar@{-}[r]&\ar@{-}[r]&\ar@{-}[r]&\ar@{-}[r]&\ar@{-}[r]&\ar@{-}[r]&\ar@{-}[r]&\\
0&1&2&3&4&5&6&7&8
}.
\end{equation}
Let $\max(\boF)\subset\ob(\Gent_R(n,p))$ be the set of objects corresponding to local maxima in
the graph $\Gamma_{\boF}$, i.e., $k\not\in\{0,n\}$ is contained in 
$\max(\boF)$ if, and only if, $\Delta_\boF$ contains a subdiagram of
the form $(\xymatrix{k-1&k\ar[r]\ar[l]&k+1})$.
Moreover, $0\in\max(\boF)$ if $(\xymatrix{0\ar[r]&1})$ is a subdiagram of $\Delta_{\boF}$,
while $n\in \max(\boF)$ if $(\xymatrix{n-1&\ar[l]n})$ is a subdiagram of $\Delta_{\boF}$.
E.g., for $\boF\in\ob(\euF_R(\Gent_R(8,p)^{\op},\Rlat))$ as in \eqref{eq:arrex1} one has that
$\max(\boF)=\{1,4,8\}$.
By $\min(\boF)$ we denote the subset of $\{1,\ldots,n-1\}$ corresponding to local minima
in the graph $\Delta_{\boF}$, i.e., $\ell\not\in\{0,n\}$ is contained in 
$\min(\boF)$ if, and only if, $\Delta_\boF$ contains a subdiagram of
the form $(\xymatrix{k-1\ar[r]&k&k+1\ar[l]})$.
E.g., for the functor $\boF\in\ob(\euF_R(\Gent_R(8,p)^{\op},\Rlat))$ as in \eqref{eq:arrex1} one has
$\min(\boF)=\{3, 7\}$. Thus by construction, one has
$|\max(\boF)|=|\min(\boF)|+1$.
The following fact is straightforward.

\begin{fact}
\label{fact:hdrk1}
Let $\boF\in\ob(\euF_R(\Gent_R(n,p)^{\op},\Rlat))$ be an $R$-lattice functor of rank $1$.
Then $\nat_R(\boF,\boS^\ell)\simeq \F$ if $\ell\in\max(\boF)$,
and $\nat_R(\boF,\boS^\ell)=0$ if $\ell\not\in\max(\boF)$.
Moreover, $\boF$ is projective if, and only if, $\min(\boF)=\emptyset$.
\end{fact}

Let $\boF\in \ob(\euF_R(\Gent_R(n,p)^{\op},\Rlat))$  be an $R$-lattice functor of rank $1$
which is not projective. Let $s(\boF)\in\max(\boF)$ be the smallest
element in $\max(\boF)$, and let $t(\boF)$ be the smallest element in $\min(\boF)$.
The projective $R$-lattice functor $\boP^{s(\boF)}$ corresponds to the diagram
obtained from the diagram $\Delta_{\boF}$ by changing all arrows between vertices 
$\alpha$ and  $\alpha+1$, $\alpha \geq t(\boF)$ to $\xymatrix{(\alpha\ar[r]&\alpha+1)}$.
Let $\boF^\wedge \in\ob(\euF_R(\Gent_R(n)^{\op},\Rlat))$  be the $R$-lattice functor of rank $1$
corresponding to the diagram
obtained from the diagram $\Delta_{\boF}$ by changing all arrows between vertices 
$\alpha-1$ and  $\alpha$, $\alpha \leq t(\boF)$ to $\xymatrix{(\alpha-1&\alpha\ar[l])}$.
E.g., for $\boF\in \ob(\euF_R(\Gent_R(8,p)^{\op},\Rlat))$ as in \eqref{eq:arrex1},
$\Gamma_{\boF^\wedge}$ is given by replacing the first segment by the path 
$\xymatrix{\ar@{..>}[r]&}$ in \eqref{eq:arrex2};
and $\Delta_{\boF^\wedge}$ is given by
\begin{equation}
\label{eq:arrex3}
\xymatrix@C.5truecm{
0&1\ar[l]&2\ar[l]&3\ar[l]&4\ar[r]\ar[l]&5\ar[r]&6\ar[r]&7&8\ar[l]
}
\end{equation}


\subsection{The global dimension of $\Gent_R(n,p)$}
\label{ss:globgent}
The following property will be essential for the subsequent analysis.

\begin{lem}
\label{lem:rank1seq}
Let $\boF\in \ob(\euF_R(\Gent_R(n,p)^{\op},\Rlat))$  be an $R$-lattice functor of rank~$1$
which is not projective. Then one has a short exact sequence of $R$-lattice functors
\begin{equation}
\label{eq:rank1seq}
\xymatrix{
0\ar[r]&\boP^{t(\boF)}\ar[r]^-{\psi}&\boP^{s(\boF)}\oplus\boF^{\wedge}\ar[r]^-{\phi}&\boF\ar[r]&0.
}
\end{equation}
\end{lem}

\begin{proof} One can identify $\boP^{s(\boF)}$ and $\boF^\wedge$ as subfunctors of $\boF$ by
putting
\begin{equation}
\label{eq:subPs}
\begin{aligned}
\boP^{s(\boF)}(k)&=
\begin{cases}
\hfil\boF(k)\hfil&\ \text{for $k\leq t(\boF)$},\\
\image(\boF(i_{k,t(\boF)}))&\ \text{for $k> t(\boF)$};
\end{cases}\\
\boF^\wedge(k)&=
\begin{cases}
\image(\boF(t_{k,t(\boF)}))&\ \text{for $k\leq t(\boF)$},\\
\hfil\boF(k)\hfil&\ \text{for $k> t(\boF)$}.
\end{cases}
\end{aligned}
\end{equation}
for $k\in\{0,\ldots,n\}$.
Let $\phi_1\colon \boP^{s(\boF)}\to \boF$ and $\phi_2\colon \boF^\wedge\to\boF$
denote the canonical inclusions. By construction, 
$\phi=\phi_1\oplus\phi_2\colon\boP^{s(\boF)}\oplus\boF^\wedge\to\boF$ is surjective
with kernel $\kernel(\phi)\subseteq \boP^{s(\boF)}\oplus\boF^\wedge$ given by
\begin{equation}
\label{eq:kerphi}
\kernel(\phi)(k)=\{\,(x,-x)\in\boP^{s(\boF)}(k)\oplus\boF^\wedge(k)
\mid x\in \boP^{s(\boF)}(k)\cap\boF^\wedge(k)\,\},
\end{equation}
i.e., $\kernel(\phi)\simeq\boP^{s(\boF)}\cap\boF^\wedge$.
By construction, $\boX=\boP^{s(\boF)}\cap\boF^\wedge$ is an $R$-lattice functor
of rank~$1$ with all maps
$\boX(t_{j,t(\boF)})$ and $\boX(i_{k,t(\boF)})$ surjective
for  $0\leq j<t(\boF)<k\leq n$. Hence all maps
$\boX(t_{j,t(\boF)})$, $\boX(i_{k,t(\boF)})$, $0\leq j<t(\boF)<k\leq n$, are isomorphisms.
Thus $\Delta_{\boX}=\Delta_{\boP^{t(\boF)}}$, and this yields the claim.
\end{proof}

The equality $|\max(\boF^\wedge)|=|\min(\boF)|+1$ has the following consequence.

\begin{prop}
\label{prop:projrank1}
Let $\boF\in \ob(\euF_R(\Gent_R(n,p)^{\op},\Rlat))$  be an $R$-lattice functor of rank~$1$.
Then one has a short exact sequence of $R$-lattice functors
\begin{equation}
\label{eq:rank1seq2}
\xymatrix{
0\ar[r]&
\textstyle{\bigoplus_{j\in\min(\boF)}}\boP^{j}\ar[r]^-{\alpha}&
\textstyle{\bigoplus_{k\in\max(\boF)}}\boP^{k}\ar[r]^-{\beta}&\boF\ar[r]&0.
}
\end{equation}
In particular, $\prdim_R(\boF)\leq 1$.
\end{prop}

\begin{proof}
We proceed by induction on $m=|\max(\boF)|$.
If $|\max(\boF)|=1$, one has $\min(\boF)=\emptyset$, and hence $\boF$ is projective.
Therefore we may assume that $m>1$,
and that the assertion is true for all $R$-lattice functors $\boG$ of rank $1$ satisfying 
$|\max(\boG)|<m$.
Let $\boF\in \ob(\euF_R(\Gent_R(n,p)^{\op},\Rlat))$ with $|\max(\boF)|=m>1$. Hence 
$|\max(\boF^\wedge)|=m-1$, and, by induction, one has a short exact sequence
\begin{equation}
\label{eq:rank1seq3}
\xymatrix{
0\ar[r]&
\textstyle{\bigoplus_{j\in\min(\boF^\wedge)}}\boP^{j}\ar[r]^-{\alpha^\wedge}&
\textstyle{\bigoplus_{k\in\max(\boF^\wedge)}}\boP^{k}\ar[r]^-{\beta^\wedge}&\boF^\wedge\ar[r]&0.
}
\end{equation}
For $s=s(\boF)$, $t=t(\boF)$ and $\beta=(\iid_{\boP^s}\oplus\beta^\wedge)\circ\phi$ one has a 
commutative and exact diagram
\begin{equation}
\label{eq:rank1seq4}
\xymatrix{
&0\ar@{..>}[d]&0\ar[d]&&\\
0\ar@{..>}[r]&\kernel(\zeta)\ar@{..>}[d]\ar@{..>}[r]&\textstyle{\bigoplus_{j\in\min(\boF^\wedge)}}\boP^{j}
\ar[d]_{\alpha^\wedge}\ar[r]&0\ar[d]&\\
0\ar[r]&\ker(\beta)\ar[r]\ar@{-->}[d]_{\zeta}
&\boP^s\oplus\textstyle{\bigoplus_{k\in\max(\boF^\wedge)}}\boP^{k}
\ar[d]_{\iid_{\boP^s}\oplus\beta^\wedge}\ar[r]^-\beta
&\boF\ar@{=}[d]\ar[r]&0\\
0\ar[r]&\boP^t\ar@{..>}[d]\ar[r]^\psi&\boP^s\oplus\boF^\wedge\ar[r]^\phi\ar[d]&\boF\ar[r]\ar[d]&0\\
&0&0&0&
}
\end{equation}
where $\psi$ and $\phi$ are as in Lemma~\ref{lem:rank1seq} and $\zeta$ is the induced map.
By the snake lemma, one may extend this diagram by the arrows ``$\xymatrix@C.3truecm{\ar@{..>}[r]&}$''. 
Hence $\ker(\beta)\simeq\bigoplus_{j\in\min(\boF)}\boP^j$, and
this yields the claim.
\end{proof}

\begin{rem}
\label{rem:gentn1}
By Remark~\ref{rem:diafun}, every functor
$\boF\in\ob(\euF_R(\Gent_R(1,p)^{\op},\Rlat))$ of rank $1$ is projective and relative injective.
\end{rem}

Finally, 
one concludes the following theorem 
which is somehow counterintuitive 
in view of Remark~\ref{rem:gentle}.

\begin{thm}
\label{thm:gldimgent}
Let $p$ be a prime number, and let $R$ be a principal ideal domain
of characteristic $0$ such that $pR$ is a prime ideal. Then
\begin{itemize}
\item[(a)]  $\glatdim_R(\Gent_R(1,p))=0$ and $\gldim_R(\Gent_R(1,p))=1$; and
\item[(b)] $\glatdim_R(\Gent_R(n,p))=1$ and $\gldim_R(\Gent_R(n,p))=2$ for $n\geq 2$.
\end{itemize}
In particular, $(\Gent_R(n,p),\sigma)$ has the Whitehead property.
\end{thm}

\begin{proof} 
Suppose that $n\geq 2$.
By Proposition~\ref{prop:projrank1}, the projective dimension of
any $R$-lattice functor of rank 1 is less or equal to $1$.
Hence by Fact~\ref{fact:rank11}, induction on the rank
and the Horseshoe lemma \cite[Lemma~2.5.1]{ben:coh1},
$\Gent_R(n,p)$ is of global $R$-lattice dimension less or equal to $1$.
Since there are $R$-lattice functors of rank 1 which are not
projective, one concludes that $\glatdim_R(\Gent_R(n,p))=1$.
For any simple functor $\boS^\ell$, $0\leq \ell\leq n$, one has
$\prdim(\boS^\ell)=2$. Thus $\gldim_R(\Gent_R(n,p))=2$.

By Remark~\ref{rem:gentn1}, any $R$-lattice functor of rank 1
of $\Gent_R(1,p)$ is projective and relative injective.
Hence by Fact~\ref{fact:rank11}, induction on the rank
and the Horseshoe lemma, any $R$-lattice functor is projective,
i.e., $\glatdim_R(\Gent_R(1,p))=0$.
For the simple functors $\boS^\ell$, $\ell\in\{0,1\}$, one has
$\prdim(\boS^\ell)=1$. Thus $\gldim_R(\Gent_R(n,p))=1$.

The final remark is a direct consequence of Fact~\ref{fact:glwh}.
\end{proof}


\subsection{Projective $R$-lattice functors}
\label{ss:projgent}
Let $R$ be a discrete valuation domain of characteristic $0$ with
maximal ideal $pR$. Then $R$ is a noetherian ring,
and every proper subfunctor $\boG\subsetneq\boF$ of a functor
$\boF\in\ob(\euF_R(\Gent_R(n,p)^{\op},\RMod^{\fg}))$ must be contained
in a maximal subfunctor $\boM\subsetneq\boF$.
Moreover, from the discussion in subsection~\ref{ss:simpgent}
one concludes that $\boF/\boM\simeq\boS^\ell$ for some $\ell\in\{0,\ldots,n\}$.
We define the {\it radical} of $\boF\ob(\euF_R(\Gent_R(n,p)^{\op},\RMod^{\fg}))$ by
\begin{equation}
\label{eq:rad}
\rad(\boF)=\bigcap_{\substack{\boM\subsetneq\boF\\ \boM\ \text{maximal}}}\boM.
\end{equation}
and the {\it head} of $\boF$ by $\hd(\boF)=\boF/\rad(\boF)$. In particular,
there exist non-negative integers $f_0,\ldots,f_n\in\N_0$ such that
\begin{equation}
\label{eq:hdgent}
\hd(\boF)\simeq f_0\boS^0\oplus\cdots\oplus f_n\boS^n.
\end{equation}
Here we used the abbreviation $m\boZ=\oplus_{1\leq j\leq m} \boZ$.
Moreover,
$\hd(\boF)=0$ if, and only if, $\boF=0$.
Furthermore, the following property holds.

\begin{fact}
\label{fact:sur}
Let $R$ be a discrete valuation domain of characteristic $0$ with
maximal ideal $pR$, and let $\phi\colon\boG\to\boF\in
\mor(\euF_R(\Gent_R(n,p)^{\op},\RMod^{\fg}))$ be a natural transformation
of functors with values in the category of finitely generated $R$-modules. 
Then $\phi$ is surjective if, and only if,
the induced map $\hd(\phi)\colon\hd(\boG)\to\hd(\boF)$ is surjective.
\end{fact}

From this one concludes the following property.

\begin{fact}
\label{fact:rankgent}
Let $R$ be a discrete valuation domain of characteristic $0$ with
maximal ideal $pR$, and let $\boF\in\ob(\euF_R(\Gent_R(n,p)^{\op},\Rlat))$.
Then $\rk(\boF)\leq\dim_{\F}(\hd(\boF))$, and equality holds if, and only if,
$\boF$ is projective.
\end{fact} 

\begin{proof}
Suppose that $\hd(\boF)\simeq f_0\boS^0\oplus\cdots\oplus f_n\boS^n$.
Put $\boP=f_0\boP^0\oplus\cdots\oplus f_n\boP^n$. Since $\boP$ is projective,
there exists a natural transformation $\phi\colon\boP\to\boF$ such that 
$\hd(\phi)\colon\hd(\boP)\to\hd(\boF)$ is an isomorphism. By Fact~\ref{fact:sur},
$\phi$ is surjective, and thus
\begin{equation}
\label{eq:rankgent}
\dim_{\F}(\hd(\boF))=\dim_{\F}(\hd(\boP))=\rk(\boP)\geq\rk(\boF).
\end{equation}
If $\rk(\boF)=\dim_{\F}(\hd(\boF))$, then $\phi$ must be an isomorphism.
Assume that $\boF$ is projective. Then $\phi$ is split-surjective, i.e., there exists a projective
$R$-lattice functor $\boQ\in\ob(\euF_R(\Gent_R(n,p)^{\op},\Rlat))$ such that
$\boP\simeq\boF\oplus\boQ$. As $\hd(\phi)$ is an isomorphism, this yields 
$\hd(\boQ)=0$. Hence $\boQ=0$, and $\boF$ is isomorphic to $\boP$.
\end{proof}

The proof of Fact~\ref{fact:rankgent} has shown also the following.

\begin{cor}
\label{cor:projgent}
Let $R$ be a discrete valuation domain of characteristic $0$ with
maximal ideal $pR$, and let $\boP\in\ob(\euF_R(\Gent_R(n,p)^{\op},\Rlat))$
be a projective $R$-lattice functor satisfying
$\hd(\boP)\simeq f_0\boS^0\oplus\cdots\oplus f_n\boS^n$.
Then $\boP\simeq f_0\boP^0\oplus\cdots\oplus f_n\boP^n$.
\end{cor}


\section{Cohomological Mackey functors for cyclic p-groups}
\label{s:cmcp}
Throughout this section we assume that $R$ is a discrete valuation domain
of characteristic $0$ with maximal ideal $pR$ and that $G$ is a finite cyclic $p$-group of order
$p^n$.

\subsection{The deflation functor}
\label{ss:unipro2}
Let $\pi\colon \mf_R(G)\longrightarrow\Gent_R(n,p)$ denote the unitary projection
(cf. \S\ref{ss:unipro1}), let
$\argu^\pi=\ifl^\pi(\dfl^\pi(\argu))$, and let $\eta\colon\iid_{\MC_G(\RMod)}\to\argu^\pi$
denote the unit of the adjunction. In particular, $\eta_\boX\colon\boX\to\boX^\pi$
is surjective, and $\eta_{\boX,G}\colon\boX_G\to\boX^\pi_G$ is an isomorphism
for all $\boX\in\ob(\MC_G(\RMod))$.

\begin{fact}
\label{fact:defcat1}
Let $G$ be a finite cyclic group of order $p^n$,
let $R$ be an integral domain of characteristic $0$,
and let $\boX\in\ob(\MC_G(\Rlat))$ be a cohomological
$R$-lattice functor which is Hilbert$^{90}$.
Then one has a canonical isomorphism
\begin{equation}
\label{eq:defcat1}
\dfl^\pi(\boX)(k)\simeq \image(t_{U,G}^\boX)\subseteq \boX_G,
\end{equation}
for $U\subseteq G$, $|U|=p^{n-k}$.
\end{fact}

\begin{proof}
Let $g\in G$ be a generator of $G$, i.e., $\dfl^\pi(\boX)(k)=\boX_U/(1-g)\boX_U$.
Periodicity of Tate cohomology implies that $\hH^{-1}(G/U,\boX_U)=H^1(G/U,\boX_U)=0$.
Thus by \eqref{eq:isolat}, $\kernel(t^{\boX}_{U,G})=(1-g)\boX_U$. Hence the induced map
\begin{equation}
\label{eq:defcat2}
\xymatrix{
\boX_U/(1-g)\boX_U\ar[r]^-{\tilde{t}^\boX_{U,G}}&\boX_G
}
\end{equation}
is injective. This yields the claim.
\end{proof}

From Fact~\ref{fact:defcat1} one concludes the following.

\begin{cor}
\label{cor:defcat}
Let $G$ be a finite cyclic group of order $p^n$,
let $R$ be an integral domain of characteristic $0$,
and let $\phi\colon\boX\to\boY\in\mor(\MC_G(\Rlat))$ be a natural transformation of cohomological
$R$-lattice functors with the following properties:
\begin{itemize}
\item[(i)] $\boX$ and $\boY$ are Hilbert$^{90}$;
\item[(ii)] $\phi_G\colon\boX_G\to\boY_G$ is injective.
\end{itemize}
Then $\dfl^\pi(\phi)\colon\dfl^\pi(\boX)\to\dfl^\pi(\boY)$ is injective. In particular, if 
\begin{equation}
\label{eq:defcat3}
\xymatrix{
0\ar[r]&
\boX\ar[r]^{\alpha}&
\boY\ar[r]^{\beta}&
\boZ\ar[r]& 0}
\end{equation} 
is a short exact sequence of $R$-lattice functors
all of which are Hilbert$^{90}$, then 
\begin{equation}
\label{eq:defcat4}
\xymatrix{
0\ar[r]&
\dfl^\pi(\boX)\ar[r]^{\dfl^\pi(\alpha)}&
\dfl^\pi(\boY)\ar[r]^{\dfl^\pi(\beta)}&
\dfl^\pi(\boZ)\ar[r]& 0}
\end{equation} 
is exact.
\end{cor}

\begin{prop}
\label{prop:defcat}
Let $G$ be a finite cyclic group of order $p^n$,
let $R$ be an Dedekind domain of characteristic $0$,
and let $\boX\in\ob(\MC_G(\Rlat))$ be a cohomological $R$-lattice functor which is 
Hilbert$^{90}$. Then $\boX$ is $\pi$-acyclic.
\end{prop}

\begin{proof}
Let $(\boP_\bullet,\partial^\boP_\bullet,\eps_\boX)$ be a projective
resolution of $\boX$ in $\MC_G(\RMod)$ with $\boP_k$ projective $R$-lattice
functors. In particular, $\boQ_k=\image(\partial_k^P)$, $k\geq 1$, is an $R$-lattice functor.
By construction, one has the short exact sequences
\begin{equation}
\label{eq:defcat5}
\xymatrix@R=3pt{
0\ar[r]&\boQ_1\ar[r]&\boP_0\ar[r]&\boX\ar[r]&0,\\
0\ar[r]&\boQ_{k+1}\ar[r]&\boP_k\ar[r]&\boQ_k\ar[r]&0,
}
\end{equation}
for $k\geq 1$. Thus by induction and Proposition~\ref{prop:cycint},
$\boQ_k$ is a Hilbert$^{90}$ $R$-lattice functors for all $k\geq 1$. 
Hence by Corollary \ref{cor:defcat}, one has
short exact sequences
\begin{equation}
\label{eq:defcat6}
\xymatrix@R=3pt{
0\ar[r]&\dfl^\pi(\boQ_1)\ar[r]&\dfl^\pi(\boP_0)\ar[r]&\dfl^\pi(\boX)\ar[r]&0,\\
0\ar[r]&\dfl^\pi(\boQ_{k+1})\ar[r]&\dfl^\pi(\boP_k)\ar[r]&\dfl^\pi(\boQ_k)\ar[r]&0,
}
\end{equation}
for $k\geq 1$. This implies that $\caL_k\dfl^\pi(\boX)=0$ for all $k\geq 1$.
\end{proof}

Let $R$ be a discrete valuation domain of characteristic $0$ with
maximal ideal $pR$, and let $G$ be a cyclic $p$-group.
As in subsection~\ref{ss:projgent} one concludes that every
proper subfunctor $\boY\subsetneq\boX$ of a cohomological
$G$-Mackey functor $\boX\in\ob(\MC_G(\RMod^{\fg}))$ must be
contained is a maximal subfunctor $\boM\subsetneq\boX$.
Therefore we define the {\it radical} of $\boX\in\ob(\MC_G(\RMod^{\fg}))$ by
\begin{equation}
\label{eq:radmac}
\rad(\boX)=\bigcap_{\substack{\boM\subsetneq\boX\\ \boM\ \text{maximal}}}\boM.
\end{equation}
and the {\it head} of $\boX$ by $\hd(\boX)=\boX/\rad(\boX)$. 
By Remark~\ref{rem:minproj}, there exist non-negative integers
$f_U\in\N_0$, $U\subseteq G$, such that 
\begin{equation}
\label{eq:hdmac}
\hd(\boX)\simeq\textstyle{\bigoplus_{U\subseteq G} f_U\boS^U}.
\end{equation}
Since every simple cohomological $G$-Mackey functor 
$\boS\in\ob(\MC_G(\RMod^{\fg}))$ is isomorphic to 
$\ifl^\pi(\bSigma)$ for some simple functor
$\bSigma\in\ob(\euF_R(\Gent_R(n,p)^{\op},\RMod^{\fg}))$,
one has $\kernel(\eta_\boX)\subseteq \rad(\boX)$.
This inclusion has the following consequence.

\begin{prop}
\label{prop:surmac}
Let $R$ be a discrete valuation domain of characteristic $0$ with maximal ideal $pR$,
and let $G$ be a finite cyclic $p$-group. Let
$\phi\colon\boX\to\boY$ be a natural transformation of cohomological
$G$-Mackey functors with values in the category $\RMod^{\fg}$. 
Then the following are equivalent.
\begin{itemize}
\item[(i)] $\phi$ is surjective;
\item[(ii)] $\phi^\pi\colon\boX^\pi\to\boY^\pi$ is surjective;
\item[(iii)] $\hd(\phi)\colon\hd(\boX)\to\hd(\boY)$ is surjective.
\end{itemize}
\end{prop}

\begin{proof}
The natural surjection $\tau\colon\iid\to\hd(\argu)$ factors through 
the natural surjection $\eta\colon\iid\to\argu^\pi$, i.e., there exists a natural surjection 
$\psi\colon\argu^\pi\to\hd(\argu)$ such that $\tau=\psi\circ\eta$.
This yields the implications (i) $\Rightarrow$ (ii) $\Rightarrow$ (iii).
Suppose that (iii) holds and that $\phi$ is not surjective. Then
$\image(\phi)$ is contained in a maximal subfunctor $\boM\subsetneq\boY$.
Thus for $\boS=\boY/\boM$, the kernel of the map
$\phi_\ast\colon\nat_G(\boY,\boS)\to\nat_G(\boX,\boS)$ is non-trivial.
However, in the commutative diagram
\begin{equation}
\label{eq:diasur}
\xymatrix{
\nat_G(\hd(\boY),\boS)\ar[r]^{\hd(\phi)_\ast}\ar[d]_{(\tau_\boY)_\ast}&
\nat_G(\hd(\boX),\boS)\ar[d]^{(\tau_\boX)_\ast}\\
\nat_G(\boY,\boS)\ar[r]^{\phi_\ast}&\nat_G(\boX,\boS)
}
\end{equation}
the vertical maps are isomorphisms, and $\hd(\phi)_\ast$ is injective
forcing $\kernel(\phi_\ast)=0$, a contradiction. This yields the claim.
\end{proof}

We are now ready to prove the following theorem which is one of the key results
in this paper.

\begin{thm}
\label{thm:defcat}
Let $R$ be a discrete valuation domain of characteristic $0$ 
with maximal ideal $pR$, 
let $G$ be a finite cyclic $p$-group,
and let $\boX\in\ob(\MC_G(\Rlat))$
be a cohomological $G$-Mackey functor
with values in the category of $R$-lattices which is Hilbert$^{90}$. 
Then there exists a finite $G$-set $\Omega$
such that $\boX\simeq\boh^0(R[\Omega])$.
In particular, $\boX$ is projective.
\end{thm}

\begin{proof}
The deflation functor
$\dfl^\pi\colon\MC_G(\RMod)\longrightarrow\euF_R(\Gent_R(p,n),\RMod)$
associated to the unitary projection $\pi\colon\mf_R(G)\longrightarrow
\Gent_R(p,n)$ has the following properties:
\begin{itemize}
\item[(1)] $(\mf_R(G)),\sigma)$ is $\circledast$-symmetric (cf. Prop.~\ref{prop:Ryon}).
\item[(2)] $\Gent_R(n,p)$ has global $R$-lattice dimension less or equal to $1$ 
(cf. Thm.~\ref{thm:gldimgent}), and thus has the Whitehead property (cf. Fact~\ref{fact:glwh}
and Fact~\ref{fact:gldim1}).
\item[(3)] An $R$-lattice functor $\boY\in\ob(\MC_G(\Rlat))$ is
$\circledast$-acyclic if, and only if, it has the Hilbert$^{90}$ property
(cf. Prop.~\ref{prop:Ryon}). By Proposition~\ref{prop:defcat},
such a functor is $\pi$-acyclic.
\end{itemize}
In particular, the hypothesis of Theorem~\ref{thm:whitehead} are satisfied, and one concludes
that $\boZ=\dfl^\pi(\boX)\in\ob(\euF_R(\Gent_R(n,p),\RMod))$ is projective.
Hence there exist non-negative integers $f_0,\ldots, f_n$ such that 
$\boZ\simeq f_0\boP^0\oplus\cdots f_n\boP^n$ (cf. Cor.~\ref{cor:projgent}).

Let $\eta_{\boX}\colon\boX\to\boX^\pi$ be the canonical map (cf. \S\ref{ss:unipro2}),
i.e., $\boX^\pi\simeq\bigoplus_{0\leq k\leq n}\ifl^\pi(f_k\boP^k)$.
Let $U_k\subseteq G$ denote the unique subgroup of $G$ of index $p^k$,
and let $\Omega$ be the $G$-set
$\Omega=\bigsqcup_{0\leq k\leq n} f_k(G/U_k)$. Put $\boP=\boh^0(R[\Omega])$.
Then $\boP\in\ob(\MC_G(\Rlat))$ is projective (cf. Fact~\ref{fact:projcohmac}).
Since $\dfl^\pi(\boh^0(R[G/U_k]))\simeq \boP^k$ for all $k\in\{0,\ldots,n\}$, one has
an isomorphism $\phi^\pi\colon\boP^\pi\to\boX^\pi$. Since $\boP$ is projective,
there exists a homomorphism of cohomological $G$-Mackey functors such that the
diagram
\begin{equation}
\xymatrix{
\boP\ar[r]^{\phi}\ar[d]_{\eta_{\boP}}&\boX\ar[d]^{\eta_\boX}\\
\boP^\pi\ar[r]^{\phi^\pi}&\boX^{\pi}
}
\end{equation}
commutes. By construction, $\phi^\pi_G$ is an isomorphism,
and $\eta_{\boP,G}$ and $\eta_{\boX,G}$ are isomorphisms (cf. \S\ref{ss:unipro2}).
Thus $\phi_G$ is an isomorphism.
In particular, with the same notations as used in subsection \ref{ss:injcrit}, the map
$\gr_0(\phi_G)\colon\gr_0(\boP)\to\gr_0(\boX)$ is an isomorphism.
By hypothesis, $\boX$ is an $R$-lattice functor with the Hilbert$^{90}$ property,
and the same is true for $\boP$ (cf. Remark~\ref{rem:projH0}).
Hence $\gr_0(\boX)$ and $\gr_0(\boP)$ are of type $H^0$ (cf. Prop.~\ref{prop:torhil}),
and $\phi$ is split-injective (cf. Prop.~\ref{prop:inj}).
Moreover, by Proposition~\ref{prop:surmac}, $\phi$ must be surjective. 
This yields the claim.
\end{proof}

As an immediate consequence of Remark~\ref{rem:projH0}
we obtain the following.

\begin{cor}
\label{cor:profperm}
Let $R$ be a discrete valuation domain of characteristic $0$ 
with maximal ideal $pR$, 
let $G$ be a finite cyclic $p$-group of order $p^n$,
and let $\boP\in\ob(\MC_G(\Rlat))$
be a projective $R$-lattice functor.
Then there exist non-negative integers $f_W\in\N_0$, $W\subseteq G$,
such that $\boP\simeq\bigoplus_{U\subseteq G} f_W\boP^W$, i.e.,
\begin{equation}
\label{eq:K0}
K_0(\mf_R(G))\simeq\burn(G)\simeq \Z^n,
\end{equation}
where $\burn(G)$ denotes the Burnside ring of $G$.
\end{cor}

In case that the $\RG$-lattice $M$ satisfies a Hilbert$^{90}$ property,
one obtains the following.

\begin{cor}
\label{cor:hil90lat1}
Let $R$ be a discrete valuation domain of characteristic $0$ 
with maximal ideal $pR$, 
let $G$ be a finite cyclic $p$-group,
and let $M$ be an $\RG$-lattice such that 
$H^1(U,\rst^G_U(M))=0$ for every subgroup $U$ of $G$.
Then there exists a finite $G$-set $\Omega$ such that
$M\simeq R[\Omega]$.
\end{cor}

\begin{proof}
By hypothesis, $\boX=\boh^0(M)$ is a cohomological
$G$-Mackey functor with values in the category of
$R$-lattices satisfying the Hilbert$^{90}$ property.
Thus by Theorem~\ref{thm:defcat}, there exists a finite
$G$-set $\Omega$ such that $\boX\simeq\boh^0(R[\Omega])$.
Hence evaluating the functors $\boX$ and 
$\boh^0(R[\Omega])$ on the subgroup $\{1\}$
yields the claim.
\end{proof}

The following property is a direct consequence of Tate duality
(cf. Prop.~\ref{prop:tatedual}) and completes the proof of
Theorem A.

\begin{prop}
\label{prop:elequi}
Let $R$ be a principal ideal domain of characteristic $0$, 
let $G$ be a finite group, let $U$ be a subgroup of $G$,
and let $M$ be an $\RG$-lattice.
Then the following are equivalent.
\begin{itemize}
\item[(i)] $H^1(U,\rst^G_U(M^\ast))=0$;
\item[(ii)] $\hH^{-1}(U,\rst^G_U(M))=0$;
\item[(iii)] $M/\omega_{R[U]}M$ is torsion free.
\end{itemize}
\end{prop}

\begin{proof}
By \eqref{eq:tatedual2}, (i) and (ii) are equivalent.
Let $N_U\colon M\to M^U$ be the $U$-norm map, i.e.,
for $m\in M$ one has $N_U(m)=\sum_{u\in U}u\cdot m$.
As $M$ is an $R[U]$-lattice, $M^U$ is an $R$-lattice.
Hence 
\begin{equation}
\label{eq:torid}
\tor_R(M/\omega_{R[U]}M)=\kernel(N_U)/\omega_{R[U]}M=\hH^{-1}(U,\rst^G_U(M)),
\end{equation}
where $\tor_R(\argu)$ denotes the $R$-submodule of $R$-torsion elements.
Thus (ii) and (iii) are equivalent.
\end{proof}


\subsection{Projective dimensions}
\label{ss:projdim}
In conjunction with Proposition~\ref{prop:cycint}, Theorem~\ref{thm:defcat} 
has strong implications on the projective dimension of a cohomological
Mackey functor of a cyclic $p$-group.

\begin{thm}
\label{thm:projdimmac}
Let $R$ be a discrete valuation domain of characteristic $0$ 
with maximal ideal $pR$, 
let $G$ be a finite cyclic $p$-group,
and let $\boX\in\ob(\MC_G(\RMod^{\fg}))$.
Let
\begin{equation}
\label{eq:parproj}
\xymatrix{
\boP_2\ar[r]^{\der_2}&\boP_1\ar[r]^{\der_1}&
\boP_0\ar[r]^{\eps_{\boX}}&\boX\ar[r]&0}
\end{equation}
be a partial projective resolution of $\boX$ by
projective $R$-lattice functors. Then
\begin{itemize}
\item[(a)] $\kernel(\der_2)$ is a projective $R$-lattice functor,
i.e., $\prdim(\boX)\leq 3$.
\item[(b)] If $\boX$ is $i$-injective, then $\kernel(\der_1)$ is a 
projective $R$-lattice functor, i.e., one has $\prdim(\boX)\leq 2$.
\item[(c)] If $\boX$ is of type $H^0$, then $\kernel(\der_0)$ is a 
projective $R$-lattice functor, i.e., one has $\prdim(\boX)\leq 1$.
\end{itemize}
In particular, if $G$ is non-trivial then
$\glatdim(\mf_R(G))= 2$, and $\gldim(\mf_R(G))= 3$.
\end{thm}

\begin{proof}
(a) By Proposition~\ref{prop:cycint}(a), (b) and (c),
$\kernel(\der_0)$ is an $R$-lattice functor and thus $i$-injective,
$\kernel(\der_1)$ is of type $H^0$, and
$\kernel(\der_2)$ is Hilbert$^{90}$.
Hence Theorem~\ref{thm:defcat} yields the claim in this case.
(b) and (c) follow by a similar argument.
From (a) one concludes that $\gldim(\mf_R(G))\leq 3$,
and (b) implies $\glatdim(\mf_R(G))\leq 2$.
If $G$ is non-trivial, the discussion in subsection~\ref{ss:cycsec} shows that
$\prdim(\boB^G)=3$. This yields the final remark (cf. \eqref{eq:dims}).
\end{proof}

\begin{rem}
\label{rem:gorfin}
Let $\F$ be a field of characteristic $p$,
and let $G$ be a non-trivial, finite cyclic $p$-group.
Then $\mf_{\F}(G)$ is not of finite global dimension,
but $\mf_{\F}(G)$ is $2$-Gorenstein (cf. Prop.~\ref{prop:gorcm}).
This phenomenon occured already for the gentle
$R$-order categories (in dimension $1$) (cf. Rem.~\ref{rem:gentle}).
\end{rem}

\subsection{Lattices}
\label{ss:latt}
From Theorem~\ref{thm:defcat} and
Theorem~\ref{thm:projdimmac}(b), one concludes the following.

\begin{thm}
\label{thm:preslat}
Let $R$ be a discrete valuation domain of characteristic $0$ 
with maximal ideal $pR$, 
let $G$ be a finite cyclic $p$-group,
and let $M$ be an $\RG$-lattice. Then
there exist finite $G$-sets $\Omega_0$ and $\Omega_1$,
and a short exact sequence
\begin{equation}
\label{eq:shexseqlat}
\xymatrix{
0\ar[r]&R[\Omega_1]\ar[r]&R[\Omega_0]\ar[r]&M\ar[r]&0.
}
\end{equation}
\end{thm} 

\begin{proof}
Let $\boX=\boh^0(M)$. As
$\boX$ is of type $H^0$, Theorem~\ref{thm:projdimmac}(b)
implies that $\boX$ has a projective resolution
\begin{equation}
\label{eq:profreslat}
\xymatrix{
0\ar[r]&\boP_1\ar[r]^{\der_1}&\boP_0\ar[r]^{\eps_\boX}&\boX\ar[r]&0,
}
\end{equation}
where $\boP_0$ and $\boP_1$ are projective $R$-lattice functors.
As $\boP_0$ and $\boP_1$ have the Hilbert$^{90}$ property (cf. Remark~\ref{rem:projH0}),
Theorem~\ref{thm:defcat} implies that there exist finite $G$-sets $\Omega_0$ and $\Omega_1$
such that $\boP_i=\boh^0(R[\Omega_i])$, $i\in\{0,1\}$.
Thus evaluating the functors on $\{1\}$ yields the claim.
\end{proof}


\subsection{Extending A.~Weiss' theorem}
\label{ss:aweiss}
The following property can be seen as an extension of A.~Weiss' theorem
for finite cyclic $p$-groups.

\begin{prop}
\label{prop:exweiss}
Let $R$ be a discrete valuation domain of characteristic $0$
with maximal ideal $pR$ for some prime number $p$, let $G$ be a finite cyclic $p$-group,
and let $M$ be an $\RG$-lattice. Suppose that for some subgroup $N$ of $G$ one has
\begin{itemize}
\item[(i)] $\rst^G_N(M)$ is an $R[N]$-permutation module;
\item[(ii)] $M^N$ is an $R[G/N]$-permutation module.
\end{itemize}
Then $M$ is isomorphic to an $\RG$-permutation module, i.e., there exists some
finite $G$-set $\Omega$ such that $M\simeq R[\Omega]$.
\end{prop}

\begin{proof}
Let $U$ be a subgroup of $G$. If $U\subseteq N$, then
by (i), $\rst^G_U(M)$ is an $R[U]$-permutation module.
Thus one has $H^1(U,\rst^G_U(M))=0$.
Suppose that $N\subsetneq U$. Since $N$ is a normal subgroup of $U$,
the $5$-term exact sequence in cohomology yields an exact sequence
\begin{equation}
\label{eq:5term}
\xymatrix{0\ar[r]&H^1(U/N,M^N)\ar[r]&H^1(U,\rst^G_U(M))\ar[r]&H^1(N,\rst^G_N(M))^{G/N}
}
\end{equation}
Hence by (i), one has $H^1(N,\rst^G_N(M))=0$.
From (ii) one concludes that $M^N$ is an $R[U/N]$-permutation module,
and therefore
$H^1(U/N,M^N)=0$. Thus by
\eqref{eq:5term} one has
$H^1(U,\rst^G_U(M))=0$.
The assertion then follows from Theorem A.
\end{proof}
\bibliography{gorenstein}
\bibliographystyle{amsplain}
\end{document}